\documentclass[a4paper,10pt]{scrartcl}
\usepackage[utf8]{inputenc}
\usepackage[english]{babel}
\usepackage[fixlanguage]{babelbib}
\usepackage[noadjust]{cite}
\usepackage{amssymb, amsmath, amstext, amsopn, amsthm, amscd, amsxtra, amsfonts}
\usepackage{mathrsfs}
\usepackage{colonequals}
\usepackage{todonotes}
\usepackage{enumerate}
\usepackage{graphicx}%
\usepackage{colonequals}
\usepackage[all]{xy}
\usepackage[normalem]{ulem}
\usepackage{hyperref}%
\hypersetup{
urlcolor = red,
colorlinks = true,
linkcolor = blue,
citecolor = blue,
linktocpage = true,
pdftitle = {Incompressible Euler: A geometric approach},
pdfauthor = {Mario Maurelli, Klas Modin, Alexander Schmeding},
bookmarksopen = true,
bookmarksopenlevel = 1,
unicode = true,
hypertexnames =false
}%

\swapnumbers
\newtheoremstyle{standard}
{16pt} 
{16pt} 
{} 
{} 
{\bfseries}
{} 
{ } 
{{\thmname{#1~}}{\thmnumber{#2.}}\thmnote{~(#3)}} 

\newtheoremstyle{kursiv}
{16pt} 
{16pt} 
{\itshape} 
{} 
{\bfseries}
{} 
{ } 
{{\thmname{#1~}}{\thmnumber{#2.}}\thmnote{~(#3)}} 

\theoremstyle{standard}
\newtheorem{defn} [subsection]{Definition}
\newtheorem{ex} [subsection]{Example}

\newtheorem{rem} [subsection]{Remark}

\newtheorem{setup} [subsection]{}

\theoremstyle{kursiv}
\newtheorem{thm}[subsection]{Theorem}
\newtheorem*{mainthm1}{First main theorem}
\newtheorem*{mainthm2}{Second main theorem}
\newtheorem{prop} [subsection]{Proposition}
\newtheorem{cor} [subsection]{Corollary}
\newtheorem{lem} [subsection]{Lemma}

\newtheorem{hp}[subsection]{Assumption}

\newcommand{\cA}{\ensuremath{\mathcal{A}}}
\newcommand{\cB}{\ensuremath{\mathcal{B}}}

\newcommand{\cF}{\ensuremath{\mathcal{F}}}
\newcommand{\cG}{\ensuremath{\mathcal{G}}}

\newcommand{\cM}{\ensuremath{\mathcal{M}}}
\newcommand{\cN}{\ensuremath{\mathcal{N}}}

\newcommand{\dd}{\mathop{}\!\mathrm{d}}

\newcommand{\R}{\ensuremath{\mathbb{R}}}

\newcommand{\N}{\ensuremath{\mathbb{N}}}

\DeclareMathOperator{\id}{id}
\DeclareMathOperator{\Diff}{Diff}
\DeclareMathOperator{\comp}{comp}
\DeclareMathOperator{\one}{\mathbf{1}}
\newcommand{\Dmu}[1][s]{\mathrm{Diff}^{#1}_\mu(K)}
\newcommand{\tr}{\mathrm{tr}}

\newcommand{\lan}{\langle}
\newcommand{\ran}{\rangle}
\newcommand{\VFsmu}[1][s]{\mathfrak{X}^{#1}_\mu (K)}
\newcommand{\VFslmu}{\mathfrak{X}^{s+\ell}_\mu (K)}

\newcommand{\Frechet}{Fr\'{e}chet }
\newcommand{\coloneq}{\colonequals}

\usepackage{soul,xcolor}

\date{}
\title{Incompressible Euler equations with stochastic forcing: a geometric approach}
\author{Mario Maurelli\footnote{Dipartimento di Matematica `Federigo Enriques', Universit\`{a} degli Studi di Milano, via Saldini 50, 20133 Milano, Italy, \href{mailto:mario.maurelli@unimi.it}{mario.maurelli@unimi.it}}, Klas Modin\footnote{Department of Mathematical Sciences, Chalmers and University of Gothenburg, Chalmers Tvargata, SE-41296, Gothenburg, Sweden, \href{mailto:klas.modin@chalmers.se}{klas.modin@chalmers.se}}, Alexander Schmeding\footnote{Department of Mathematics, Universitet i Bergen, All\'{e}gate 41, 5020 Bergen, Norway \href{mailto:alexander.schmeding@uib.no}{alexander.schmeding@uib.no}.}}

\begin{document}

\maketitle

\begin{abstract}
We consider a stochastic version of Euler equations using the infinite-dimensional geometric approach as pioneered by Ebin and Marsden \cite{EM70}. For the Euler equations on a compact manifold (possibly with smooth boundary) we establish local existence and uniqueness of a strong solution in spaces of Sobolev mappings (of high enough regularity). Our approach combines techniques from stochastic analysis and infinite-dimensional geometry and provides a novel toolbox to establish local well-posedness of stochastic non-linear partial differential equations. 
\end{abstract}


\textbf{MSC2010:} 
35Q31 (primary);  
60H15, 
76B03, 
58D15, 
58B25 
\\[2.3mm]

\textbf{Keywords:}
stochastic Euler equation, half-Lie group, manifold of Sobolev\\ mappings, Ebin-Marsden theory, stochastic integration on Hilbert manifolds  
\newpage
\tableofcontents

\section*{Introduction and main result}

Since formulated in 1757~\cite{Eu1757}, Euler's equations for the motion of an inviscous incompressible fluid have had a profound role in science; in geophysics, in meteorology, in aerospace engineering, in astrophysics, and, of course, in mathematics where advanced techniques for existence and uniqueness in 2D and 3D provide important mathematical tools and new theoretical insights.

To accommodate external influence for which a precise model is missing, it is natural to consider stochastic versions of the Euler equations.
The main result in this paper is a new framework for local existence and uniqueness of stochastic nonlinear partial differential equations (PDEs) of hydrodynamic type\footnote{Although we here consider stochastic versions of the classical Euler equation, the framework is general enough for the wider class of \emph{Euler--Arnold equations}, c.f.\ Section~\ref{sect:above_beyond}.} evolving on compact manifolds. 
A cornerstone in the analysis of deterministic hydrodynamic PDEs is the infinite-dimensional geometric theory developed in 1970 by Ebin and Marsden~\cite{EM70}.
Our framework constitutes an extension of this theory to stochastic PDEs.

We introduce noise as \emph{stochastic forcing}.
That is, the noise is an external fluctuating force acting on the fluid particles. The force is assumed to have Gaussian distribution, uncorrelated in time but correlated in space: the latter condition means roughly that, at any time, nearby fluid particles should experience nearly the same force. 
As a first case, we consider here \emph{additive noise} in the forcing, i.e.\ noise not depending on the solution itself.
We leave the more general case of \emph{multiplicative noise} for future investigation.

Early works on Euler equations with stochastic forcing are mostly in the two-dimensional (2D) case \cite{Bes1999,BesFla1999,BrzPes2001}. The three-dimensional (3D) case has been treated in \cite{MR1747615,Kim2009, Glatt-Holtz-Vicol-2014}. In particular, in \cite{Glatt-Holtz-Vicol-2014} local existence and uniqueness among smooth solutions in 2D and 3D domains is proved for a wide class of noises.
Let us also mention that, complementary to stochastic forcing, \emph{transport noise} offers a different way to introduce noise in the Euler equations: here one considers the vorticity formulation and noise is added in the Poisson equation relating the stream function and the vector field, see e.g.\ \cite{CFH17,MR3325187,BrzFlaMau2016}. Other types of noises are also possible, like non-Gaussian noises (for example L\'evy noises) or not time-uncorrelated noises (for example rough paths). 

In the aforementioned work, the analysis is based on PDE techniques combined with stochastic analysis.
Our approach is different; it is based on the infinite-dimensional geometric technique first devised by Arnold~\cite{Ar1966}, who discovered that solutions to the deterministic Euler equations can be interpreted as geodesic curves on the infinite-dimensional configuration manifold of volume preserving diffeomorphisms equipped with a right-invariant Riemannian metric.
Ebin and Marsden~\cite{EM70} thereafter used Arnold's geometric viewpoint to obtain local well-posedness of the deterministic equations, including smooth dependence on intial conditions.
As the streamlined presentation in \cite{Eb2015} shows, their strategy is to prove that the Lagrangian formulation of the Euler equations, as a second order system on the tangent bundle of a Hilbert manifold of diffeomorphisms, is a smooth, infinite-dimensional ordinary differential equation (ODE). Once this is achieved, standard Picard iterations yield the local well-posedness (since the finite-dimensional ODE analysis extends to Banach spaces). Let us highlight here that this approach is not just a formal way to view PDEs as equations in infinite dimension: this ODE on the Hilbert manifold is driven by a \emph{smooth} velocity field and not by an unbounded operator.

Extending the Ebin and Marsden framework to a stochastic setting requires us to deal with stochastic differential equations (SDEs) on infinite-dimensional manifolds. A theory for SDEs, in the case of Hilbert manifolds, appears in \cite{Elw82}. It is based on Stratonovich integration, a stochastic integral invariant under change of charts, and can be extended to Banach manifolds, cf.\ \cite{BaE01}. 
Another approach to stochastic integration on manifolds is developed in \cite{DalBel1989}, which is more based on It\^o integration and needs a strong Riemannian structure on the (infinite-dimensional) manifold. Applications of SDEs on infinite-dimensional manifolds to mathematical physics and stochastic PDEs have been considered, for example in the case of SDEs on loop manifolds (e.g.\ \cite{BaE01}), manifold-valued stochastic PDEs (e.g.\ \cite{Fun1992,BrzOnd2013,Hus2015}), SDEs on infinite-dimensional Lie groups (e.g.\ \cite{AlbDal2000}). The lifting to an infinite-dimensional manifold of diffeomorphisms has been used in \cite{Elw1978,BrzElw1996} to show the existence of stochastic flows for finite-dimensional SDEs.
However, none of these obtain a stochastic extension of the Ebin and Marsden result as we accomplish in this paper.

Closer to our framework is instead an interesting approach to deterministic, viscous PDEs like Navier-Stokes equations. Roughly speaking, the idea is that the Laplacian at the Eulerian level corresponds to noise at the Lagrangian level, provided an average (over the noise) is taken. An example is given by Gliklikh in \cite[Chapter 16]{Glik11} (based on previous works like \cite{BelGli2002}): in particular, the Lagrangian SDEs \cite[\S 16.24]{Glik11} and \cite[\S 16.25]{Glik11} are similar to our Lagrangian SDE \eqref{eq:euler_formal_lagrangian} (more precisely, to \eqref{Euler:stochastic_Lagrangian}).
However, in those equation the noise is taken on the velocity (not on the force) and mean derivatives are used for the velocity equation, so that the Eulerian counterpart becomes the deterministic Navier-Stokes equations (and not stochastic Euler equations).
Also, Gliklikh uses an It\^o-like formulation (which needs a connection at the infinite-dimensional level, something avoided here). This line of research has been developed by Cruzeiro and coauthors \cite{CruSha2009,ArnCruFan2018}. In particular, the recent paper \cite{Cru2019} also considers stochastic Euler equations (with transport noise), where the Eulerian-to-Lagrangian link is shown without going into the infinite-dimensional analysis.
Another important contribution to the geometric viewpoint on stochastic Euler equations is given in \cite{CaM08,CaFaM07}: these papers take Arnold's viewpoint and, in 2D, construct a solution to Euler equations via Girsanov transform (a transformation of the SDE that removes the drift), though no direct analysis of the infinite-dimensional Lagrangian SDE is given.

In summary, the notion of using the Ebin and Marsden approach to stochastic PDEs has been considered by many authors, but to fully develop such an analysis has remained an open problem.
Here we solve this problem in a first setting: stochastic Euler equations with additive noise. As a byproduct,
we get local existence and uniqueness for stochastic Euler equations on any compact manifold (of any finite dimension $d\ge 2$). To the best of our knowledge this has not
been obtained before.

We stress that our purpose is not to improve on previous existence results\footnote{The existence results in \cite{Glatt-Holtz-Vicol-2014}, based on traditional SPDE techniques on $\mathbb{R}^n$, are sharper than what presented here in terms of regularity and in terms of classes of noise than what we obtain in this paper. It is possible to sharpen our results in terms of regularity and we believe we can treat more general types of noises. However, a refined analysis easily becomes technical and would diverge from the essence of this paper. Instead, we chose to present the most accessible results to illustrate our framework. We postpone refined analysis to future, more specialized publications.}.
Rather, we anticipate the framework presented here to be the first step in a stochastic analog of the Ebin and Marsden framework for deterministic non-linear PDEs, which for 50 years has been a valuable complement to more traditional PDE analysis, and which for various equations has led to deep insights:
Fredholmness of solution maps~\cite{EbMiPr2006,MiPr2010}, global existence results~\cite{Mi2002,MiMu2013}, vanishing geodesic distances~\cite{MiMu2005,BaHaPr2020}, limits between compressible and incompressible Euler equations~\cite{Eb1975}, averaged Euler equations and sharp results on viscous limits~\cite{MaRaSh2000,MaSh2003,CoSh2007}, stability of steady solutions~\cite{Ar1965,FrSh2001,Pr2004}, connections between fluid dynamics and optimal transport~\cite{MR2456522}, improved numerical methods~\cite{MoVi2020}, etcetera.
The key point is that a stochastic Ebin and Marsden framework enables stochastic calculus on manifolds as long as the techniques one use are oblivious to the manifold being modeled on a Banach (or Hilbert) space.
One such technique is local existence of stochastic ordinary differential equations on Hilbert and Banach manifolds~\cite{Elw82,BaE01}, as we elevate in this paper.
In analogous ways, our framework may also lead to progress on related stochastic topics.
Some of these, not addressed in this paper, include:
(i) A different, geometric approach to stationary measures for stochastic Euler and Navier-Stokes equations, developed e.g. in ~\cite{Ku2012}.
(ii) Stability analysis for stochastic Euler equations (with possibly multiplicative noise): the interplay between stability/completeness/explosion and geometry for a stochastic system has been investigated in the finite dimensional case for example in \cite{Aze1974,Li1994}.
(iii) Metastability (i.e. the notion where, for small noise and large time, the solution ``jumps'' from a stable steady solution to the deterministic equation to another one), see e.g.~\cite{BouLauZab2014} for 2D Euler equation, one could try to study the transition paths between steady states from a geometric viewpoint.
(iv) Interpretations of the stochastic Euler equations (with added damping and an potential) as a Hamiltonian Monte-Carlo method on an infinite-dimensional Riemannian manifold~\cite{GiCa2011}.
(v) Structure preserving numerical methods for stochastic Euler equations, inspired by the deterministic setting in~\cite{MoVi2020}.

We now continue the introduction by presenting the stochastic Euler model.
Thereafter we state the main result.
The fluid domain is a compact oriented Riemannian manifold $K$ of dimension $d$, possibly with smooth boundary, and the equations we study are (formally) given by
\begin{equation}\label{eq:euler_formal}
\left\{
\begin{aligned}
	& \frac{\partial u}{\partial t} + \nabla_u u + \nabla p = \dot W \\
	& \operatorname{div}(u) = 0
\end{aligned}\right.
\end{equation}
where $u$ is a vector field on $K$ of Sobolev regularity $H^s$ describing the fluid velocity, $p$ is the pressure function, $\nabla_u$ denotes the co-variant derivative along $u$, and the vector valued noise $\dot W$ corresponds to a fluctuating external force field; more precisely, $W$ is a Wiener process with values in the space of Sobolev, divergence free vector fields.
If $K$ has a boundary, we require $u$ and the noise field $\dot W$ to be tangential to it.


The next step is to reformulate equation~\eqref{eq:euler_formal} using the Lagrangian variable $\Phi$, with $\dot\Phi = u\circ\Phi$.
The Lagrangian equation then takes the form
\begin{equation}\label{eq:euler_formal_lagrangian}
\left\{
\begin{aligned}
& \nabla_{\dot{\Phi}}\dot{\Phi} + \nabla p\circ\Phi = \dot W\circ\Phi \\
& \operatorname{div}(\dot\Phi\circ\Phi^{-1}) = 0 .
\end{aligned}\right.
\end{equation}
Here, we interpret the flow $\Phi$ as evolving on the infinite-dimensional manifold of volume preserving diffeomorphisms of Sobolev regularity $H^s$.
Using geometric and stochastic analysis on infinite-dimensional manifolds we prove the following result.



\begin{mainthm1}
	Fix $s>d/2+1$ and suppose that the noise takes values in the space of $H^{s+2}$ divergence-free (and tangential) vector fields.
	Then local existence and uniqueness hold for the Lagrangian formulation \eqref{eq:euler_formal_lagrangian}.
\end{mainthm1}

Relating this result to the original Eulerian formulation~\eqref{eq:euler_formal} is more complicated than in the deterministic case because of the stochastic terms.
That is why the following theorem requires higher regularity.

\begin{mainthm2}
	Fix $s>d/2+4$ and suppose that the noise takes values in the space of $H^{s+2}$ divergence-free (and tangential) vector fields. 
	Then local existence and uniqueness for the stochastic Euler equation~\eqref{eq:euler_formal} hold.
\end{mainthm2}

We expect that the regularity required from noise terms and initial data is not optimal. Indeed, regularisation arguments (cf.\ Section~\ref{sect:above_beyond}) should lead to much lower regularity assumptions. 
The point of the paper is to develop, as neatly as possible, a stochastic version of the Ebin and Marsden results.
For this reason we do not deal with optimal regularity questions here. Instead we plan show in a separate paper how to close the regularity gap between our two main theorems above, getting local well-posedness for the stochastic Euler equation~\eqref{eq:euler_formal} for $s>d/2+1$.

The paper is organized as follows: In Section~\ref{sect: SDE:HMfd}, we give a self-contained presentation of Stratonovich SDEs on infinite-dimensional Hilbert manifolds. This is complemented by extensive background material on stochastic integration on infinite-dimensional spaces in Appendix \ref{App:stochastic}. In Section~\ref{sect:VF:Sobo:diff}, we introduce the Hilbert manifold of Sobolev diffeomorphisms on $K$ and analyze the regularity properties of the composition map. This is also complemented by background material on spaces of Sobolev maps in Appendix~\ref{app: Sobolev}. In Section~\ref{sect: EM-SEuler}, we state and prove the main results: we prove the local existence and uniqueness for the Lagrangian formulation and then we show, first formally and then rigorously, the link between the Lagrangian form and the Eulerian form, concluding the local existence and uniqueness for the latter. In Section~\ref{sect:above_beyond}, we discuss possible extensions and future developments; as another example of application of our methods, we show local well-posedness of averaged Euler equation \eqref{eq:averaged_euler_stochastic}.

 
\paragraph*{Notation and Conventions}
 Let $E,F$ be Banach spaces and $U\subseteq E$ open. A map $f \colon E \supseteq U \rightarrow F$ is said to be of class $C^k$ if it is $k$ times continuously \Frechet differentiable. We write $Df$ for its derivative and say that a $C^k$-map is of class $C^{k,1}$ (or $C^{k,1}_{\text{loc}}$) if its $k$th derivative is (locally) Lipschitz continuous.
 Spaces of linear operators $L(E,F)$ are endow with the operator norm $\|\cdot\|_{L(E,F)}$. If there is no possible confusion, we write $\|\cdot\|_{L}$ for the operator norm.
 
 We let $K$ be a compact manifold (possibly with smooth boundary $\partial K$). 
 Further, $H,H'$ etc.\ will be separable Hilbert spaces and $M$ a (possibly infinite-dimensional) metrisable and separable manifold modelled on $H$. Our main example will be  $M=\Diff_\mu^s(K)$, the group of volume preserving diffeomorphisms of class $H^s$ on $K$. 
 For a differentiable map $f \colon M \rightarrow N$ we let $Tf$ be its tangent map. 
 Generic charts will be denoted by $\kappa \colon U\rightarrow \R^d$ and by $\psi \colon U\rightarrow H$. 
 For a a vector field $X$ we let $X^\kappa$ either be its representative in a chart $\kappa$, or (abusing notation) its principal part $\text{pr}_2\circ T\kappa \circ X \circ \kappa^{-1}$.  
 
 Concerning the stochastic setting, we denote the 
 Borel-$\sigma$-algebra of a topological space $T$ by $\mathcal{B}(T)$ and let $\one_A$ be the indicator function of the set $A$. 
 Moreover, we diverge from usual notation and denote Stratonovich integrals by $\int f \bullet dW$. The usual ''$\circ$'' will be reserved for composition.
\medskip

\noindent \textbf{Acknowledgements} We would like to thank D.~Holm, F.~Flandoli, Z.~Brze\'zniak, I.~Bailleul, P.\ Harms and M.\ Bauer for helpful discussions on the subject of this work; we thank Z.~Brze\'zniak also for providing useful references on the topic. Moreover, we thank H.~Gl\"ockner for providing information which led to Lemma~\ref{lem: expo:lin}.
The work was supported by EU Horizon 2020 grant No 691070, and by the Swedish Research Council (VR) grant No 2017-05040. Part of this work was undertaken when M.M.\ was at the University of York, supported by the Royal Society via the Newton International Fellowship NF170448 ``Stochastic Euler equations and the Kraichnan model''. A.S.\ was supported by the Einstein foundation while conducting work at TU Berlin.

\section{Stochastic differential equations on Hilbert manifolds}\label{sect: SDE:HMfd}

We assume familiarity with the basic objects of probability theory and stochastic processes such as Brownian motion. 
This section will review Stratonovich integration on Hilbert manifolds. We take the main results (with small modifications), from, and follow the approach in, \cite{Elw82,BaE01}, taking also some facts from \cite{daPaZ} and \cite{BrNeVeWe2008}. Moreover, Appendix \ref{App:stochastic} includes a self contained review of stochastic integration and stochastic differential equations on Hilbert spaces.


We will introduce SDEs on Hilbert manifolds and obtain the main local well-posedness result. This will enable us to establish the local well-posedness of Euler flows in Section~\ref{sect: EM-SEuler}. The Stratonovich integral is the right type of integral to use here: it is invariant under change of charts, because the It\^o formula for Stratonovich integrals is analogue to the classical chain rule, without second order terms. This invariance is at the basis of the so-called Malliavin's ``transfer principle'': quoting \cite{Eme1990}, `geometric constructions involving manifold-valued curves can be extended to manifold-valued [stochastic] processes by replacing classical calculus with Stratonovich stochastic calculus'. On the contrary, the It\^o integral instead would not be invariant under change of charts (cf.\ Appendix \ref{App:stochastic}).  Let us now describe the basic setting used throughout this section.

\begin{setup}\label{setup:filtration_standard}
Fix a probability space $(\Omega,\cA,P)$ together with a filtration $\cF=(\cF_t)_t$, $t\in [0,+\infty[$, that is a non-decreasing (i.e. $\cF_s\subseteq \cF_t$ for every $s\le t$) family of $\sigma$-algebras contained in $\cA$, indexed by $t\in [0,+\infty[$.
We assume that $\cA$ is the $\sigma$-algebra $\cF_\infty$ generated by all $\cF_t$ and that $\cF$ is complete and right-continuous (these are technical and classical assumptions).
Further, we let $L^p(\Omega)$ be the space of real valued $L^p$-functions on $\Omega$. For $X$ in $L^1(\Omega)$, the symbol
\begin{align*}
E[X] = \int_\Omega X dP
\end{align*}
denotes the expectation under $P$. For a random variable or process with values in a topological space, measurability and progressive measurability are understood with respect to the Borel $\sigma$-algebra on the topological space, unless differently specified.
\end{setup}

\begin{setup}
Let $H,E$ be separable Hilbert space. We choose an $E$-valued $Q$-Brownian motion $W$ with respect to $\cF$ (cf.\ \ref{setup:QWiener} for the definition of Brownian motion). Let now $M$ be a metrizable separable differentiable (i.e.\ $C^\infty$) manifold modelled on the Hilbert space $H$.\footnote{We are mostly interested in the case where $M = \Diff^s_\mu (K)$ is the manifold of $H^s$-diffeomorphisms preserving a volume form $\mu$ and $H = \VFslmu$ is the space of divergence free vector fields (for $s$ large enough) and tangent to the boundary $\partial K$, cf.\ Appendix \ref{app: Sobolev}.} Let $(\psi_\alpha \colon O_\alpha \rightarrow V_\alpha)$ be a countable atlas of $C^\infty$ charts, with $\psi_\alpha$ and $\psi_\alpha^{-1}$ and their first $100$ derivatives bounded (we have chosen $100$ just as high enough number for what we need). Here ``bounded'' refers for the mappings to the natural metric on manifold and model space, and for the derivatives with respect to the natural operator norms induced by the Hilbert space norm and the norm on the tangent spaces induced by a strong Riemannian metric.
\end{setup}

\begin{setup}
Let $\mathbb{A}$ and $\mathbb{B}$ be vector bundles over $M$ whose fibres we denote by $A_x$ and $B_x$. Recall from \cite[Lemma 1.2.12]{MR1330918} the associated vector bundle of linear maps 
\begin{align*}
L(\mathbb{A};\mathbb{B}) := \bigcup_{x\in M} \{x\} \times L(A_x, B_x).
\end{align*}
A trivialisation of $L(\mathbb{A};\mathbb{B})$ over a chart $(U,\psi)$ of $M$ is then given by $L(\varphi_\mathbb{A},\varphi_\mathbb{B})(x,F) := (\psi (x),\varphi_\mathbb{B} (x,F(\varphi_\mathbb{A}^{-1} (\psi(x),\cdot))))$, where $\varphi_\mathbb{A}$ and $\varphi_\mathbb{B}$ are bundle trivialisations over $(U,\psi)$. Note that, since we are not considering maximal solutions here, it will suffice to work in a fixed bundle trivialisation. If one of the vector bundles is trivial with typical fibre $E$, we write (if $\mathbb{A}$ is trivial) $L(E;\mathbb{B})$ to shorten the notation. As usual we denote the space of $C^k$-sections, i.e.\ $C^k$ maps $f:M\rightarrow L(\mathbb{A};\mathbb{B})$ with $\pi(f)=\id$, by $\Gamma_{C^k} (L(\mathbb{A};\mathbb{B}))$.

Let now $f \colon M\rightarrow \tilde{H}$ be a $C^{k+1}$ map with values in a separable Hilbert space $\tilde{H}$ and $Df := \text{pr}_2 \circ Tf \colon TM \rightarrow \tilde{H}$ the principal part of the tangent map.
Note that $Df$ induces a $C^k$-bundle morphism over the identity $(\pi, Df) \colon TM \rightarrow M\times \tilde{H}$, where $\pi \colon TM \rightarrow M$ is the bundle projection. Hence postcomposition yields a $C^k$-bundle map $(Df)_* \colon L(\mathbb{A};TM) \rightarrow L(\mathbb{A};\tilde{H})$.\footnote{Working in local trivialisations (cf.\ e.g.\ \cite[III, \S 4]{MR1666820}) we may assume $M\subseteq H$ open. Thus the $C^k$-property follows immediately from smoothness of operator composition, \cite[I, \S 2, Proposition 2.6]{MR1666820}, and the \Frechet derivative $Df \colon M \rightarrow L(H,\tilde{H})$ being a $C^k$-map.} Thus for $\sigma \in \Gamma_{C^k} (L(E;TM))$ we obtain a $C^k$-map $Df_*(\sigma) \colon M\rightarrow L(E,\tilde{H})), m \mapsto (Df)_*(\sigma(m))$.
\end{setup}

Differently from the case of Hilbert space, we do not have a notion of It\^o differential that we can use to define a Stratonovich integral.
However, we can give the notion of solution to a stochastic differential equation on $M$:
\begin{align}
\begin{aligned}\label{eq:SDE_man}
&dX_t = b(X_t)dt +\sigma(X_t) \bullet dW_t,\\
&X_0 = \zeta.
\end{aligned}
\end{align}
Here the drift $b\colon M\rightarrow TM$ and the diffusion coefficient $\sigma:M\rightarrow L(E;TM)$ are given sections assumed continuous and in $C^1$ resp., the initial datum $\zeta\colon\Omega\rightarrow M$ is a $\cF_0$-measurable random variable.

\begin{defn}\label{def:SDE_man_sol}
An $M$-valued, progressively measurable process $X$, defined on $[0,\tau)$ for some accessible stopping time $\tau$ with $\tau>0$ $P$-a.s., is called a local strong solution to \eqref{eq:SDE_man} if it has $P$-a.s. continuous paths and, for every $C^2$ function $h \colon M\rightarrow \tilde{H}$ with values in a separable Hilbert space $\tilde{H}$, there holds, $P$-a.s.,
\begin{equation}\begin{aligned}
h(X_t) = h(\zeta) +&\int_0^t Dh \circ b(X_r) dr +\int_0^t Dh \circ \sigma(X_r) dW_r \\+&\frac12 \int_0^t \tr[D(Dh \circ \sigma) \circ \sigma(X_r)Q] dr, \quad \forall t\in[0,\tau).\end{aligned} \label{eq:SDE_man_def}
\end{equation}
Here the trace is understood as
\begin{align*}
\tr[D(Dh \circ \sigma) \circ \sigma(X_r)Q] = \sum_k D(Dh\circ \sigma) \circ \sigma(X_r)Q^{1/2}e_k \circ Q^{1/2}e_k
\end{align*}
with $\{e_k\}_k$ a complete orthonormal basis of $E$.
\end{defn}
This definition extends to progressively measurable processes $X$, defined on the closed interval $[0,\tau]$ for some $P$-a.s. finite stopping time $\tau$ with $\tau>0$ $P$-a.s., requiring that \eqref{eq:SDE_man_def} holds for all $t$ in $[0,\tau]$. Moreover, if $X$ takes values $P$-a.s. in an open subset $U$ of $M$, it is enough that $b$ and $\sigma$ are defined and in $C^0$, in $C^1$ resp.\ on $U$.
If $M=H$, thanks to Theorem \ref{thm:Ito_formula_Strat} in Appendix \ref{App:stochastic}, the above definition is equivalent to the definition of solution on Hilbert spaces.

Given a solution $X$ to \eqref{eq:SDE_man} and a $C^1$ map $g\colon M\rightarrow L(E,\tilde{H})$, with $\tilde{H}$ a separable Hilbert space, we define the Stratonovich integral
\begin{align*}
\int_0^t g(X_r) \bullet dW_r = \int_0^t g(X_r) dW_r +\frac12 \int_0^t \tr[Dg \circ \sigma(X_r) Q] dr,\quad t\in[0,\tau)
\end{align*}
(or for $t\in[0,\tau]$ if $X$ is defined on $[0,\tau]$). By this definition, the Stratonovich integral depends a priori on $X$, $g$ and $\sigma$ separately, though also here one can show that the definition depends only on $g(X)$ (see Remark \ref{rem:Strat_def} in Appendix \ref{App:stochastic}).

The above definitions and properties are extended to Stratonovich differentials, that is the case of a more general drift $B$, namely
\begin{align}
dX_t = B_t dt +\sigma(X_t) \bullet dW_t\label{eq:Strat_diff_manifold}
\end{align}
where $B$ is a $TM$-valued progressively measurable process, with $P$-a.s. continuous paths, requiring that, $P$-a.s., $\pi(B_t) = X_t$ for every $t$ (we recall that $\pi$ is the bundle projection on the base point). For example, a progressively measurable process $X$, defined on $[0,\tau)$ for some accessible stopping time $\tau$, satisfies \eqref{eq:Strat_diff_manifold} if it has $P$-a.s. continuous paths, it satisfies $P$-a.s. $\pi(B_t)=X_t$ for every $t$ and, for every $C^2$ function $h\colon M\rightarrow \tilde{H}$ with values in a separable Hilbert space $\tilde{H}$, there holds, $P$-a.s.,
\begin{align}
\begin{aligned}\label{eq:Strat_diff_manifold_def}
h(X_t) = h(X_0) +&\int_0^t Dh \circ B_r dr +\int_0^t Dh \circ \sigma(X_r) dW_r \\+&\frac12 \int_0^t \tr[D(Dh \circ \sigma) \circ \sigma(X_r)Q] dr, \quad \forall t\in[0,\tau).
\end{aligned}
\end{align}

\begin{thm}[It\^o formula for manifold-valued processes]\label{thm:Ito_formula_manifold}
Assume the above setting. Let $X$ be a solution to \eqref{eq:SDE_man} and let $f:M\rightarrow N$ be a $C^2$ function, where $N$ is a metrizable separable differentiable manifold modelled on a (separable) Hilbert space $H'$. Then there holds, $P$-a.s.,
\begin{align}\label{eq:Ito_formula_manifold}
f(X_t) = f(X_0) +\int_0^t Tf \circ B_r dr +\int_0^t Tf \circ \sigma (X_r) \bullet dW_r,\quad \forall t\in[0,\tau).
\end{align}
The meaning of \eqref{eq:Ito_formula_manifold} is understood rigorously as follows: For every $C^2$ function $h \colon N\rightarrow\tilde{H}$ with values in a separable Hilbert space $\tilde{H}$, there holds for all $t\in [0,\tau)$, $P$-a.s.,
\begin{align}
\begin{aligned}\label{eq:Ito_formula_manifold_rig}
h\circ f(X_t) = h\circ f(X_0) +&\int_0^t Dh \circ Tf \circ B_r dr +\int_0^t Dh \circ Tf \circ \sigma(X_r) dW_r  \\ +&\frac12 \int_0^t \tr[D(Dh \circ Tf \circ \sigma) \circ \sigma(X_r)Q] dr.
\end{aligned}
\end{align}
\end{thm}

\begin{proof}
The result follows from the definition of Stratonovich differential, precisely formula \eqref{eq:Strat_diff_manifold_def}, with $h$ replaced by $h\circ f$, noting that $D(h\circ f) = Dh \circ Tf$.
\end{proof}

\begin{rem}
If $X$ takes values $P$-a.s. in an open subset $U$ of $M$, then equation \eqref{eq:Strat_diff_manifold_def} holds $P$-a.s. also for every $f \in C^2 (U, \tilde{H})$, that is, $X$ solves \eqref{eq:Strat_diff_manifold} also as $U$-valued process, where $U$ inherits the manifold structure from $M$. We argue by localisation and take for each $n$,
\begin{align*}
U_n = \{x \in M \mid \text{dist}(x,U^c)>1/n\} \subseteq U,
\end{align*}
where $\text{dist}$ is the distance induced by the metric on $M$. Now $M$ is paracompact and modelled on a Hilbert space. As Hilbert spaces admit smooth bump functions, cf.\ \cite[16.16.\ Corollary]{MR1471480}, the usual partition of unity argument shows that for every $n$ there is a $C^2$ (even smooth) $\psi^n\colon M\rightarrow \R$ with $\psi^n=1$ on $U_n$ and $\psi^n=0$ on $U_{n+1}^c$. Define $\tau^n$ as the minimum of $\tau$ and the first exit time of $X$ from $U_n$. Then for any $f \colon U\rightarrow \tilde{H}$, we apply \eqref{eq:Strat_diff_manifold_def} to $\psi^n \cdot f $ (trivially extended to a $C^2$ function on $M$), and get \eqref{eq:SDE_man_def} for $f$ before time $\tau^n$. Letting $n$ go to $\infty$, we get \eqref{eq:Strat_diff_manifold_def} for $f$ before $\tau$.
\end{rem}

As a consequence, we get the invariance of the equation under diffeomorphism:

\begin{lem}\label{lem:invariance_SDE_diffeo}
Let $U$, $V$ be open sets resp. on $M$, $N$, with $N$ another metrizable separable Hilbert manifold (modelled possibly on another separable Hilbert space), let $\varphi\colon U\rightarrow V$ be a $C^2$ diffeomorphism. Let $X$ be a solution on $[0,\tau)$ (or on $[0,\tau]$) to \eqref{eq:SDE_man} such that $X$ takes values in $U$ $P$-a.s.. Then $Y=\varphi(X)$ is a solution to
\begin{align*}
dY = T\varphi\circ b\circ \varphi^{-1}(Y) dt +T\varphi\circ \sigma\circ \varphi^{-1}(Y) \bullet dW.
\end{align*}
\end{lem}

\begin{proof}
The result follows from It\^o formula \eqref{eq:Ito_formula_manifold_rig}, provided that there holds, for any $C^2$ function $h \colon N\rightarrow\tilde{H}$ (with $\tilde{H}$ separable Hilbert space),
\begin{align*}
\tr[D(Dh \circ T\varphi \circ \sigma) \circ \sigma(X_r)Q] = \tr[D(Dh \circ T\varphi\circ \sigma\circ \varphi^{-1}) \circ T\varphi\circ \sigma\circ \varphi^{-1}(Y_r)Q].
\end{align*}
But this follows from
\begin{align*}
D(Dh \circ T\varphi\circ \sigma) = D(Dh \circ T\varphi\circ \sigma \circ \varphi^{-1} \circ \varphi) = D(Dh \circ T\varphi\circ \sigma\circ \varphi^{-1}) \circ T\varphi.
\end{align*}
The proof is complete.
\end{proof}

\begin{rem}\label{rem:invariance_SDE_map}
As one sees from the proof, we can relax the assumption that $\varphi$ is a $C^2$ diffeomorphism, requiring instead the following condition: $\varphi\colon U\rightarrow V$ is $C^2$ and there exist a continuous section $\tilde{b}\colon V\rightarrow TN$ and a $C^1$ section $\tilde{\sigma}\colon V\rightarrow L(E;TN)$ such that, for every $x$ in $U$, $\tilde{b}\circ \varphi(x) = T\varphi \circ b(x)$ and $\tilde{\sigma}\circ \varphi(x) = T\varphi \circ \sigma(x)$. In this case, $Y=\varphi(X)$ satisfies
\begin{align*}
dY = \tilde{b}(Y) dt +\tilde{\sigma}(Y) \bullet dW.
\end{align*}
We can also relax the assumptions on $b$ and $\tilde{b}$, requiring that $b$ and $\tilde{b}$ are Borel sections and $b(X)$ and $\tilde{b}(Y)$ coincide $P$-a.s. and are continuous in time $P$-a.s.
\end{rem}

As a consequence, taking $\varphi=\psi_\alpha$ (where $(\psi_\alpha)$ is a countable atlas of smooth charts), we get the expression of the SDE \eqref{eq:SDE_man} in chart and the invariance of the solution under change of chart. In the following, for any $\alpha$, we call
\begin{align*}
& b^\alpha\colon  V_\alpha\rightarrow H, \quad b^\alpha = D\psi_\alpha \circ b \circ \psi_\alpha^{-1},\\
& \sigma^\alpha\colon  V_\alpha\rightarrow L(E,H), \quad \sigma^\alpha = D\psi_\alpha \circ \sigma \circ \psi_\alpha^{-1}.
\end{align*}

\begin{cor}[Invariance under change of chart]\label{cor:change_chart}
Let $X$ be a $M$-valued process on $[0,\tau)$ (or $[0,\tau]$) such that, for some $\alpha$, $X$ takes values in $O_\alpha$ $P$-a.s.. Then $X$ solves \eqref{eq:SDE_man} if and only if $X^\alpha_t = \psi_\alpha(X_t)$ solves
\begin{align}
\begin{aligned}\label{eq:SDE_man_chart}
& dX^\alpha_t = b^\alpha(X^\alpha_t) dt +\sigma^\alpha(X^\alpha_t) \bullet dW_t,\\
& X^\alpha_0 = \psi_\alpha(\xi).
\end{aligned}
\end{align}
In particular, if, for some $\alpha$ and $\beta$, \eqref{eq:SDE_man_chart} holds and $X$ takes values in $O_\alpha\cap O_\beta$ $P$-a.s., then \eqref{eq:SDE_man_chart} holds also with $\beta$ in place of $\alpha$.
\end{cor}

Now we give the main local well-posedness result:

\begin{thm}\label{thm:local_abstract_SDE_result}
Assume that there exists an open set $U$ in $M$ and an index $\alpha$, with $\bar{U} \subseteq O_\alpha$, such that $X_0=\xi$ is in $U$ $P$-a.s., $b^\alpha$ is in $C^{0,1}$ on $\psi_\alpha(\bar{U})$ and $\sigma^\alpha$ is in $C^{1,1}$ on $\psi_\alpha(\bar{U})$. Then existence and uniqueness up to the first exit time from $U$ hold for \eqref{eq:SDE_man}, that is, for every $T>0$: there exists a solution $X$ on $[0,\tau_U\wedge T]$, where $\tau_U$ is the exit time of $X$ from $U$, of the SDE \eqref{eq:SDE_man} and, if $\tilde{X}$ is another solution defined on $[0,\tilde{\tau}]$, then $\tilde{X}=X$ on $[0,\tilde{\tau}\wedge\tau_U\wedge T]$ $P$-a.s.; moreover $\tau_U>0$ $P$-a.s..
\end{thm}

\begin{rem}
In Theorem \ref{thm:local_abstract_SDE_result} the requirements on the coefficients are formulated with respect to a manifold chart. To define intrinsically a Lipschitz section (independent of the chart), an auxiliary structure, like a strong Riemannian metric, is needed, cf.\ \cite{JaL14} for the finite-dimensional case. Working in charts we avoid a lengthy discussion or the strengthening of the requirements on the coefficients by requiring e.g.\ more orders of differentiability.  
\end{rem}

\begin{proof}[Proof of Theorem \ref{thm:local_abstract_SDE_result}]
The result follows from existence and uniqueness on Hilbert spaces, Theorem \ref{thm:wellpos_SDE_Strat}, via the equivalence between the SDE \eqref{eq:SDE_man} and its expression in chart \eqref{eq:SDE_man_chart}. Precisely, $\psi_\alpha(U)$ is an open bounded subset of $H$ (with $\psi_\alpha(\bar{U})=\bar{\psi_\alpha(U)}$), $b^\alpha$, $\sigma^\alpha$ are resp. $C^{0,1}$, $C^{1,1}$ on $\psi_\alpha(U)$ and $\psi_\alpha(\xi)$ is in $\psi_\alpha(U)$ $P$-a.s.. Hence, by Theorem  \ref{thm:wellpos_SDE_Strat}, there exists a (unique) solution $X^\alpha$ to the SDE \eqref{eq:SDE_man_chart}, on $[0,\tau]$, where $\tau$ is the exit time of $X^\alpha$ from $\psi_\alpha(U)$, and $\tau>0$ $P$-a.s.. Then, by Corollary \ref{cor:change_chart}, $X=\psi_\alpha^{-1}(X^\alpha)$ is a solution to \eqref{eq:SDE_man} on $[0,\tau]$ and $\tau$ is also the exit time of $X$ from $U$, which proves existence. Again by Corollary \ref{cor:change_chart}, for any other solution $\tilde{X}$ to \eqref{eq:SDE_man} on $[0,\tilde{\tau}]$, $\tilde{X}^\alpha=\psi_\alpha(\tilde{X})$ satisfies \eqref{eq:SDE_man_chart}, hence, by Theorem  \ref{thm:wellpos_SDE_Strat}, it coincides with $X^\alpha$ on $[0,\tilde{\tau}\wedge\tau]$ $P$-a.s., and so $X$ and $\tilde{X}$ coincide $P$-a.s., that is uniqueness. The proof is complete.
\end{proof}

\begin{rem}\label{rmk:man_randm_drift}
The invariance under diffeomorphism Lemma \ref{lem:invariance_SDE_diffeo}, Remark \ref{rem:invariance_SDE_map} and Corollary \ref{cor:change_chart} can be extended (with the same proof) to the case of a random drift, under the following assumption: given an accessible stopping time $\tau$ and an open set $U$ in $M$, the drift $b\colon [0,\tau)\times\Omega \times U \rightarrow TM$ is such that
\begin{itemize}
\item for every $x$ in $U$, $b(\cdot,\cdot,x)$ is progressively measurable,
\item it holds $P$-a.s.: for every $t$ in $[0,\tau)$, $b(t,\omega,\cdot)$ is a section on $U$, and
\item it holds $P$-a.s.: $(t,x)\mapsto b(t,\omega,x)$ is continuous.
\end{itemize}
Precisely, in Remark \ref{rem:invariance_SDE_map}, $\tilde{b}\colon [0,\tau)\times\Omega \times V \rightarrow TN$ satisfy the above assumptions on $V$ and it holds $P$-a.s.: for every $t$, $\tilde{b}(t,\omega,\cdot)\circ \varphi = T\varphi \circ b(t,\omega,\cdot)$.

Moreover, as in the flat case, the existence and uniqueness Theorem \ref{thm:local_abstract_SDE_result} can be extended to the case of a random drift, under the following assumption: given an accessible stopping time $\tau$ and an open set $U$ as in Theorem \ref{thm:local_abstract_SDE_result}, the drift $b\colon [0,\tau)\times\Omega\times U \rightarrow TM$ is such that
\begin{itemize}
\item for every $x$ in $\psi_\alpha(\bar{U})$, $b^\alpha(\cdot,\cdot,x)$ is progressively measurable,
\item it holds $P$-a.s.: for every $t$ in $[0,\tau)$, $b(t,\omega,\cdot)$ is a section on $U$,
\item it holds $P$-a.s.: $(t,x)\mapsto b^\alpha(t,\omega,x)$ is continuous, and
\item it holds $P$-a.s.: for every $t$ in $[0,\tau)$, $b^\alpha(t,\omega,\cdot)$ is Lipschitz continuous on $\psi_\alpha(\bar{U})$, uniformly with respect to $(t,\omega)$.
\end{itemize}
The proof is analogous, applying Theorem \ref{thm:wellpos_SDE_Strat} together with Remark \ref{rmk:Strat_random_drift}.
\end{rem}

\section{Vector fields on Sobolev diffeomorphisms}\label{sect:VF:Sobo:diff}

In this section, we establish the geometric setting in which we will solve the stochastic differential equations. Thus we leave the stochastic considerations of the last section for the time being and consider the geometry of the Hilbert manifold of Sobolev diffeomorphisms preserving a volume form. See also \cite{Glik11} for an introduction to the topic.
Our aim will be to construct certain second order vector fields on this manifold. This construction drives our later investigation, as the vector fields are crucial ingredients in the formulation of the stochastic differential equations we aim to investigate. 

\begin{setup}
All manifolds considered will be assumed to be smooth. Recall from \cite[Section 1]{MR2954043} that a $d$-dimensional manifold $K$ has smooth boundary, if it is locally homeomorphic to open subsets of $\overline{\R}^{d}_+ \coloneq \{ (x_1,\ldots, x_d) \in \R^d\mid x_d \geq 0\}$.
We fix a compact $d$-dimensional Riemannian manifold $(K,g_K)$ (possibly with $C^\infty$-boundary)\footnote{In case the boundary $\partial K$ is non-empty, note that the Riemannian metric turns $\partial K$ into a totally geodesic submanifold \cite[Lemma 6.4]{EM70}, i.e.\ a geodesic originating at a boundary point $k$ whose initial derivative is in $T_k \partial K$ stays in the boundary for all time.}.
 Denote by $\mu$ the volume form associated to the metric $g_K$. Further, $(N,g_N)$ will be another Riemannian manifold (possibly also with boundary).  
\end{setup}


Spaces of mappings of Sobolev type between (subsets of) Euclidean space are well studied in the literature dealing with partial differential equations, see e.g.\ \cite{MR1163193,MR0248880}. The corresponding notion for manifold valued mappings is also classical, but much less well known \cite{IKT13,EM70,MR0248880,MR0198494}.

\begin{defn}\label{defn: nonlinear:Sobolev}
Fix $s > \tfrac{d}{2}$, where $d$ is the dimension of the compact manifold $K$ (possibly with boundary). Then a continuous map $f \colon K \rightarrow N$ between manifolds is \emph{locally of class $H^s$} around $k\in K$, if there exist a pair of charts $(\kappa, \lambda)$ around $k$ and $f(k)$ such that $\lambda \circ  f \circ \kappa^{-1}$ makes sense and is a mapping of Sobolev class $H^s$, whose distributional derivatives up to order $s$ are in $L^2$.
Then we define the space 
$$H^s (K,N) \coloneq \{f \colon K \rightarrow N \mid  f \text{ is locally of class } H^s \text{ for every } k \in K \}$$
of all $H^s$-Sobolev maps. We recall in Appendix \ref{app: Sobolev} that this space can be turned into a Hilbert manifold modelled on spaces $H^s (E)$ of Sobolev sections for certain vector bundles $\pi_E \colon E \rightarrow K$. For $E=TM$ we write $\mathfrak{X}^s (K) :=H^s (TM)$.  
\end{defn}

\begin{setup}[Warning]\label{warning} On a manifold, a map is of class $H^s$ if it is everywhere locally of class $H^s$. Note that the notion of locally being $H^s$ incorporates a boundedness concept on the derivatives which do not have intrinsic meaning on a manifold (without specifying a Riemannian structure). In particular, a map locally of class $H^s$ in some pair of charts might fail to be of class $H^s$ in another pair of charts, see \cite[3.1]{IKT13}.
\end{setup}

\begin{setup}\label{setup:DMU}
From now on we choose an $s> \frac{d}{2}+1$ and define the group of volume preserving $H^s$-Sobolev diffeomorphisms
$$\Dmu \coloneq \{g \in H^s (K,K) \mid g \text{ bijective with } g^{-1} \in H^s (K,K), g^*\mu = \mu\}.$$
For simplicity, we shall also assume without further notice that $s\in \N$. Though most of the results will generalise also to fractional Sobolev spaces (cf.\ Section~\ref{sect:above_beyond} for a discussion), the integer assumption allows us to conveniently cite most results needed in our approach.

It is well known (compare \cite{EM70} or Appendix \ref{app: Sobolev}) that $\Dmu$ is a Hilbert manifold modelled on the space of divergence free $H^s$-vector fields $\VFsmu$.
Now $\Dmu$ is a topological group for which the right multiplication operator is smooth (a so called half Lie group \cite{MR3836195}), that is here $\Phi\mapsto \comp(\Phi,\phi)$ is smooth for every $\phi$. Moreover, $\Dmu$ is a metrizable manifold such that every component is separable. For our purpose separability of components is sufficient (albeit we asked the whole manifold to be separable in Section~\ref{sect: SDE:HMfd}) since solutions of stochastic equations are continuous, whence they evolve in one component if the initial conditions live in one component.
\end{setup}

\begin{setup}
If $K$ is a manifold with smooth boundary, the model space of $\Dmu$ consists of all divergence free $H^s$-vector fields which are tangential to the boundary. While it is important to have the condition in the presence of boundary, it is of no consequence for the arguments we are about to develop. Hence for the rest of the article, we will suppress the boundary condition in our notation. Thus for $K$ (with or without boundary) we shall simply write $\VFsmu$ and assume that elements in this space are tangential to the boundary.
As is explained in Appendix \ref{app: Sobolev}, the group $\Dmu$ has the same properties we discussed in \ref{setup:DMU} in the boundary-less case.
\end{setup}
 
To generate a second order equation which corresponds to the Euler equation, we construct a vertical vector fields from elements of the tangent space at the identity. Following \cite[Section 11]{EM70}, we can combine this field with the metric spray to obtain the desired second order vector field (cf.\ \cite[IV, \S 3]{MR1666820}). Let us recall some facts on $\Dmu$ and its tangent bundle from Appendix~\ref{app: Sobolev}.

\begin{setup}
One can identify the tangent space at $\Phi \in \Dmu$ as follows:
$$T_\Phi \Dmu = \{X \circ \Phi \mid X \in \VFsmu \} = T_{\id}R_{\Phi} (\VFsmu)$$
Thus $T_{\id}R_{\Phi} \colon \VFsmu \rightarrow T_\Phi \Dmu$ is a continuous linear isomorphism (with inverse $T_{\Phi}R_{\Phi^{-1}})$) and we obtain a trivialisation of the tangent bundle via the homeomorphism\footnote{As $J$ is not even $C^1$ (due to $\Dmu$ being just a half-Lie group), it is \emph{not} a diffeomorphism.}
$$J\colon T\Dmu \rightarrow \VFsmu \times \Dmu ,\quad T_{\Phi} \Dmu \ni \eta_\Phi \mapsto (\eta_\Phi \circ \Phi^{-1}, \Phi).$$
The inverse of $J$ is given by the (continuous) \emph{composition map} 
$$\text{comp} \colon \VFsmu \times \Dmu \rightarrow T\Dmu , (V,\Phi) \mapsto V\circ \Phi.$$ 
\end{setup}
 
We stress here that the lack of differentiability of the maps $J$ and $\text{comp}$ is caused by the lack of differentiability of left composition, cf.\ \ref{setup: HL}. However, the composition allows us to extend elements of the tangent space at the identity to vector fields on $\Dmu$.

\begin{defn}\label{defn:assomap}
Let $V$ be a divergence-free $H^s$ vector field on $K$ (i.e.\ an element of $\VFsmu$). Define a continuous vector field $\bar{V} \in C(\Dmu,T\Dmu)$ on $\Dmu$ as
\begin{align*}
\bar{V}(\Phi) = \text{comp}(V,\Phi).
\end{align*} 
Furthermore, every such vector field yields a continuous map 
$$B_V \colon T\Dmu \rightarrow T\Dmu,\quad \eta \mapsto \bar{V} (\pi(\eta)),$$
where $\pi \colon T\Dmu \rightarrow \Dmu$ is the bundle projection.
\end{defn}

Due to \ref{setup: HL} the section $\bar{V}$ and thus also $B_V$ will be of class $C^k$ if $V \in H^{s+k}_\mu$.
By construction $B_V$ is fibre-preserving, i.e.\ $\pi(B_V (U)) = \pi (U), \forall U\in T\Dmu$. 
Thus we can construct a vector field on $T\Dmu$ from $B_V$ by the vertical lift \cite[p.55]{KMS}.

\begin{setup}\label{setup:VLIFT}
Recall that the collection $VT\Dmu \coloneq \bigcup_{z \in T^2\Dmu} \text{ker} T_z \pi$ is a vector subbundle of $T^2 \Dmu \rightarrow T\Dmu$, called the vertical bundle. In local coordinates (see e.g.\ \cite[X, \S 4]{MR1666820} for a detailed discussion), the vertical bundle is given by elements of the form $((x,v),(0,w))$. Recall that the vertical lift $$\text{vl}_{T{\Dmu}} \colon T\Dmu \oplus T\Dmu \rightarrow VT\Dmu, (v_x,u_x) \mapsto \left.\frac{d}{d t}\right|_{t=0} (v_x +tu_x),$$
is given, in a pair $(T\psi, T\psi), T^2\psi$ of charts, by $i_{\text{vert}} ((x,v),(x,u)) \coloneq ((x,v),(0,u))$, cf.\ e.g.\ \cite[6.11]{KMS}. Now $\text{vl}_{T{\Dmu}} \colon T\Dmu \oplus T\Dmu \rightarrow VT\Dmu$ is a smooth bundle isomorphism.
Since $B_V$ from Definition \ref{defn:assomap} is fibre-preserving, we obtain a vector field 
$$E_V \colon T\Dmu \rightarrow T^2\Dmu, E_V \coloneq \text{vl}_{T\Dmu} \circ (\id,B_V).$$
By construction $E_V$ is continuous and a $C^k$-vector field if $V \in H^{s+k}_\mu$. 
\end{setup}

We will now represent $E_V$ in local charts. Fix a manifold chart $\psi \colon \Dmu \supseteq O \rightarrow \psi (O) \subseteq \VFsmu$ such that the following properties are satisfied:
\begin{enumerate}
\item $\id \in O$ and $\psi (\id) = 0$ (the zero vector field)
\item $T_{\id} \psi \colon \VFsmu = T_{\id} \Dmu \rightarrow \VFsmu$ is the identity operator
\end{enumerate}
Let us now compute the representative $E^\psi_V \coloneq T^2\psi \circ E_V \circ T\psi^{-1}$ of $E_V$ in the chart $\psi$.
\begin{equation}\label{eq: incharts}
\begin{aligned}
E^\psi_V(\Phi,\eta) &= T^2\psi \circ \text{vl}_{T\Dmu} (T\psi^{-1} (\Phi,\eta) , B_V T\psi^{-1} (\Phi,\eta)) \\
 &= i_{\text{vert}}\circ (T\psi \oplus T\psi ) (T\psi^{-1} (\Phi,\eta) , B_V T\psi^{-1} (\Phi,\eta)) \\
 &= (\Phi,\eta, 0, T \psi \circ \comp (V (\pi \circ T\psi^{-1} (\Phi,\eta)) \\
 &= (\Phi,\eta, 0, T_{\psi^{-1} (\eta)} \psi (\comp (V, \psi^{-1} (\Phi)) 
\end{aligned}\end{equation}
where $\Phi \in \psi(O) \subseteq \Diff_\mu^s (K)$. Thus for $V \in \VFslmu$, the principal part of the representative is determined by the $C^\ell$-mapping 
\begin{align}\label{eq:def:ekappal}
e^{\psi,\ell} \colon \VFslmu \times \psi (O) \rightarrow \VFsmu,\quad e^{\psi,\ell} (V , \Phi ) \coloneq  T_{\psi^{-1}(\eta)} \psi (\comp (V , \psi^{-1} (\Phi)).\end{align}
The key result to obtain regularity will be a Lipschitz estimate for $e^{\psi,\ell}$ and its derivatives. We will deduce these properties from the following technical lemma:

\begin{lem}\label{lem: expo:lin}
Let $\mathsf E,\mathsf F,\mathsf G$ be Banach spaces and $U \subseteq \mathsf E$ an open subset. Assume that $F \colon U \times \mathsf F \rightarrow \mathsf G$ is a $C^{k,1}_{\text{loc}}$-map such that 
for every $x\in U$ the map $F(x,\cdot) \colon \mathsf F \rightarrow \mathsf G$ is linear.
Then the map $F^\wedge \colon U \rightarrow L(\mathsf F,\mathsf G), x \mapsto F(x,\cdot)$ makes sense and is a mapping of class $C^{k,1}_{\text{loc}}$ (i.e.\ of class $C^k$ with $k$th derivative being locally Lipschitz-continuous).
\end{lem}

\begin{proof}
Let us note first that since $F$ is a mapping of class $C^{k,1}_{\text{loc}}$, the composition $F\circ c$ of $F$ with any $C^{k,1}_{loc}$-curve $c \colon \R \rightarrow U \times \mathsf E$ is again a curve of class $C^{k,1}_{loc}$. A mapping with this property is called $\mathcal{L}ip^k$ map (cf.\ \cite[Section 12]{MR1471480} or \cite{MR961256} for a detailed discussion). Clearly $F^\wedge \colon U \rightarrow L(\mathsf F,\mathsf G), x \mapsto F(x,\cdot)$ makes sense and by \cite[Theorem 4.3.5]{MR961256} $F^\wedge$ is again of class $\mathcal{L}ip^k$. Since $F^\wedge$ is $\mathcal{L}ip^k$, we can iteratively apply \cite[Theorem 12.8]{MR1471480} to see that all iterated directional derivatives $d^\ell F^\wedge \colon U \times \mathsf E^\ell \rightarrow L(\mathsf F,\mathsf G)$ for $\ell \leq k$ exist and are continuous, so $F^\wedge$ is a mapping of class $C^{k-1}$ \cite[Lemma A.3.3]{MR2952176}. 

To see that $F^\wedge \colon U \rightarrow L(\mathsf F,\mathsf G)$ is actually a $C^{k,1}_{\text{loc}}$ mapping we work with the directional derivatives. Recall from \cite[Proposition A.3.2]{MR2952176} that (as $F^\wedge$ is $C^{k-1}$) $F^\wedge$ will be a $C^k$-map if we can show that the mapping 
$$D^k F^\wedge \colon U \rightarrow L^k (\mathsf E,L(\mathsf F,\mathsf G)) , x \mapsto d^kF^\wedge (x;\cdot)$$ 
is continuous. 
To see this, we exploit that due to the construction of $F^\wedge (x) = F(x,\cdot)$, the $k$th directional derivative satisfies 
\begin{equation}\label{eq: relderiv}\begin{aligned}
d^k F^\wedge (x;v_1,\ldots, v_k) &= d^k F(x,\cdot; (v_1,0) , \ldots ,(v_k,0))\\
								 &= d_1^k F(x,\cdot; v_1, \ldots, v_k), \quad  \forall x\in U, \text{ and } v_1, \ldots , v_k \in E,
								 \end{aligned}
\end{equation}
where $d_1^k$ denotes the $k$th iterated partial derivative with respect to the first component of $F$. Since $F$ is of class $C^{k,1}_{\text{loc}}$, the partial derivative (cf.\ \cite[Proposition 3.5]{MR1666820}) $D^k_1 F \colon U \times \mathsf F \rightarrow L^k (\mathsf E ,\mathsf G), (x,y) \mapsto d^k_1 F (x,y;\cdot)$ is locally Lipschitz. Furthermore, an inductive argument easily shows that $D^k_1 F (x,\cdot)$ is linear for every fixed $x$.
Using \eqref{eq: relderiv} we thus observe that $D^kF^\wedge$ can be expressed as follows
\begin{align*}
(D^k_1F)^\wedge \colon U \rightarrow L(\mathsf F,L^k (\mathsf E, \mathsf G) \cong L^k (\mathsf E, L(\mathsf F,\mathsf G)), x \mapsto D^k_1 F(x,\cdot) = d^k_1 F (x,\cdot) = D^k F^\wedge.
\end{align*}
As a consequence, the computation shows that $F^\wedge$ will be of class $C^{k,1}_{\text{loc}}$ if $(D^k_1F)^\wedge$ is locally Lipschitz continuous. To this end, we note that \eqref{eq: relderiv} implies that $D^k_1 F$ is locally Lipschitz continuous as $F$ is a $C^{k,1}_{\text{loc}}$-map.
Thus for $(x_0,0) \in U \times \mathsf F$ there is $R:=R(x_0)>0$ and $L:=L(x_0) >0$ such that for all $x,y\in B_R(x_0)$ and $\lVert v\rVert , \lVert w\rVert < R$ we have 
$$\lVert D^kF(x,v) - D^kF (y,w) \rVert_{L^k (E,G)} \leq L \max \{\lVert x-y\rVert , \lVert v-w\rVert\}.$$
For $v\in \mathsf F\setminus \{0\}$ we define now $\overline{v} := \frac{R}{2} \frac{v}{\lVert v\rVert}$ and see that for $x,y \in B_R (x_0)$ we have
\begin{align*}\lVert D^kF(x,v) - D^kF (y,v) \rVert_{L^k (\mathsf E,\mathsf G)} &= \frac{2}{R} \lVert v\rVert \lVert D^kF(x,\overline{v}) - D^kF (y,\overline{v}) \rVert_{L^k (\mathsf E,\mathsf G)} \\
 &\leq \frac{2L}{R} \lVert v\rVert \lVert x-y\rVert .
\end{align*}
We conclude that $(D^kF)^\wedge$ is indeed locally Lipschitz continuous as the operator norm can be estimated as:
$\lVert (D^kF)^\wedge (x) - (D^kF)^\wedge (y)\rVert_{L(\mathsf F,L^k(\mathsf E,\mathsf G))} \leq \frac{2L}{R}\lVert x-y\rVert$.
\end{proof}

\begin{rem} (a)
It is essential that the map $F$ from Lemma \ref{lem: expo:lin} is $k$-times differentiable with $k$th derivative being locally Lipschitz. Weakening the Lipschitz assumption to mere continuity, the statement of Lemma \ref{lem: expo:lin} becomes false as \cite[12.13. Smolyanov’s Example]{MR1471480} shows. However, the converse statement (i.e.\ that $F$ is $C^k$ if $F^\wedge$ is $C^k$) is a standard result \cite[Proposition 3.10]{MR1666820} which does not hinge on Lipschitz continuity of the $k$th derivative.
\\
(b) Combining \cite[Proposition 4.3.16 and Theorem 4.3.27]{MR961256} one can deduce that between (open sets of) Banach spaces a $\mathcal{L}ip^k$-mapping is automatically of class $C^{k,1}_{\text{loc}}$. Using this result, the proof of Lemma \ref{lem: expo:lin} could have been considerably shorter. The reason we did not use this is that the cited results hinge on the following statement: A map $f\colon E\supseteq U \rightarrow F$ from a normed space to a locally convex space is $\mathcal{L}ip^0$ if and only if the mapping is locally Lipschitz \cite[Lemma 12.7]{MR1471480}. This is false if $F$ is non-normable as a counterexample due to H.\ Gl\"ockner shows (cf.\ the errata of \cite{MR1471480}). Studying the proof it was unclear to us whether the result holds for normed $F$ (this was established in \cite[Theorem 1.4.2]{MR961256} with essentially the same proof). Thus we chose to err on the side of caution and have avoided using these results.   
\end{rem}

We can now deduce from Lemma \ref{lem: expo:lin} the regularity of $e^{\psi,\ell}$.
Due to linearity of the bundle trivialisation $T\psi$ of $T\Dmu$ the $C^\ell$-map 
$$e^{\psi,\ell} \colon \VFslmu \times \psi (O) \rightarrow \VFsmu,\quad e^\psi (V , \eta ) =  T_{\psi^{-1}(\eta)} \psi (\comp (V , \psi^{-1} (\eta)).$$ 
is linear in $V$. We can thus deduce from Lemma \ref{lem: expo:lin} the following proposition (see also \cite[Chapter VIII Section 1]{Elw82} and \cite[Corollary 5.10]{BrzElw1996} for a similar result):

\begin{prop}\label{prop:VFisgood}
The $C^{\ell-1,1}_{\text{loc}}$ map $e^{\psi,\ell}$ (see \eqref{eq:def:ekappal}) gives rise to a $C^{\ell-1,1}_{\text{loc}}$-map 
$$(e^{\psi,\ell})^\wedge \colon \psi (O) \rightarrow L(\VFslmu , \VFsmu).$$
\end{prop}

We finally note that the mapping $e^{\psi,\ell}$ can not be expected to be of class $C^{\ell,1}_{\text{loc}}$. The reason for this is that $e^{\psi,\ell}$ essentially is given by the composition map of the half-Lie group $\Diff^s_\mu (K)$ and it is a folklore fact that composition with Sobolev vector fields of class $H^{s+\ell}$ is only $C^{\ell}$ but not $C^{\ell,1}_{\text{loc}}$.\footnote{If the composition were Lipschitz, the solution map of the Euler equation would be uniformly continuous, which is false, cf.\ \cite[Theorem 2.1]{HaM10}. We thank G.\ Misio\l{}ek for pointing this out.}

\section{Ebin-Marsden theory for the stochastic Euler equation}\label{sect: EM-SEuler}

In this section we combine the stochastic and geometric considerations developed in the last sections to obtain existence and uniqueness results for a stochastic version of the Euler equation for an incompressible fluid on a manifold.

\begin{setup}
As in Section~\ref{sect:VF:Sobo:diff}, we fix the following data: a compact (oriented) manifold $K$ (possibly with boundary), a Riemannian metric with associated volume form $\mu$ and a pressure function $p \in H^{s+1} (K,\R)$.
Now the \emph{classical Euler equation for an incompressible fluid} occupying $K$ is
\begin{equation}\label{Euler:classic}
\begin{cases}
\dot{u} + \nabla_u u = - \nabla p\\
\operatorname{div} u_t = 0 \text{ and } u_t \text{ tangential to } \partial K\\
u_0 \in \mathfrak{X}^s_\mu(K) 
\end{cases}
\end{equation}
Moreover, we can consider the \emph{Euler equation with (deterministic) forces}, where we add to the right hand side of \eqref{Euler:classic} a forcing term $f \colon \R \times K \rightarrow TK$, $\pi \circ f (t,k) = k, \forall (t,k)$ of suitable regularity, i.e.\ $f$ should be continuous and of class $H^{s+2}$ with respect to $K$.
\end{setup}
Following an idea by Arnold \cite{Ar1966}, one can rewrite \eqref{Euler:classic} as an ordinary differential equation on the Hilbert manifold $\Dmu$. To do this, recall the following formal facts: First, if we consider $\Dmu$ as manifold embedded into $L^2(K;K)$, then, for any $\Phi$ in $\Dmu$, $T_\Phi \Dmu=\{V\circ \Phi\mid V \in \mathfrak{X}^s_\mu(K)\}$ is orthogonal to $\nabla g \circ \phi$ for any function $g\colon K\rightarrow \R$. Second, given a manifold $M$, a vector field $B$ on $T^2M$ is the geodesic spray, that is the curve $\gamma$ satisfying $\dot(\gamma,\dot\gamma) = B(\gamma,\dot\gamma)$ is a geodesic, if and only if the velocity component of $B(V)$ is $V$ and the acceleration component is orthogonal to $T_{\pi(V)}M$, for every $V$ in $TM$.

Note that, if $u$ satisfies the Euler equation, and $\Pi$ is the Leray projection on the divergence-free vector fields, then
\begin{align*}
(I-\Pi)[\nabla_u u] +\nabla p =0.
\end{align*}
Hence, if $\Phi$ is the flow solution to $\dot\Phi = u(t,\Phi)$, then $\Phi(t,\cdot)$ is measure-preserving (because $u$ is divergence-free), hence $\Phi$ is a curve on $\Dmu$.

\begin{setup}\label{setup:Eulerforce}
To treat the Euler equation with forces as a second order equation on the infinite-dimensional manifold $\Dmu$, we augment $f \in \VFsmu[s+2]$ to a right invariant vector field which we then vertically lift to a second order vector field with values in $T^2\Dmu$. Thus we consider the following map 
$$\mathcal{V}_f \colon \R \times T\Dmu \rightarrow T^2\Dmu , \mathcal{V}_f (t, V_\Phi) := \text{vl}_{T\Dmu} (f(t,\cdot) \circ \Phi).$$
Then one modifies the geodesic spray $B$ by defining $B_f := B + \mathcal{V}_f$, \cite[\S 11]{EM70}. By the chain rule, the second-order vector field $B_f$ satisfies
\begin{align*}
\frac{\dd}{\dd t}(\Phi,\dot\Phi) &= (\dot\Phi,-\nabla\rho(\Phi) +f(\Phi))  = (\dot\Phi,(I-\Pi)[\nabla_{\dot{\Phi}\circ\Phi^{-1}} \dot{\Phi}\circ\Phi^{-1}]) + (0,\comp(f,\Phi)) \\ &=: \tilde{B}(\Phi,\dot\Phi) +\mathcal{V}_f (\dot \Phi),
\end{align*}
Now the term $\tilde{B}(\eta)$ has velocity component $\eta$ and its acceleration component is orthogonal to $T_{\pi(\eta)}\Dmu$, hence $\tilde{B}$ coincide with the geodesic spray $B$.
\end{setup}
 The equation
\begin{align}\label{Euler:Lagrangian}
\frac{\dd}{\dd t}(\Phi,\dot\Phi) = B(\Phi,\dot\Phi) +\mathcal{V}_f(\dot \Phi) = B_f (\Phi, \dot \Phi),
\end{align}
is then called the \emph{Euler equation with forces in Lagrangian form}; we will call the classical Euler equation \eqref{Euler:classic} \emph{Euler equation in Eulerian form}.

In their seminal paper \cite{EM70}, Ebin and Marsden established smoothness of the spray of the $L^2$-metric, which is not automatic due to the metric being a weak Riemannian metric. Hence the Euler equation in Lagrangian form \eqref{Euler:Lagrangian} is an ordinary differential equation (in infinite dimension) with smooth drift, opposed to many standard PDEs, and so it is solvable, at least locally, by standard Banach manifold techniques.

Here we consider the classical Euler equation \eqref{Euler:classic} with an additive noise of the form $\dot{W}(t,k)$, Gaussian, white in time and smooth in space: we take the stochastic Euler equation with (deterministic and stochastic) forces in Eulerian form
\begin{align}
\begin{cases}\label{Euler:stochastic}
\partial_t u +\nabla_u u +\nabla p = f+  \dot{W}(t,k),\\
\operatorname{div} u = 0,
\end{cases}
\end{align}
where $W$ is a Brownian motion with values in a suitable function space and $f$ is as above a suitably regular forcing term. Hence the corresponding stochastic Euler equations in Lagrangian form reads formally
\begin{align}\label{Euler:stochastic_Lagrangian}
d(\Phi,\dot\Phi) = (B(\Phi,\dot\Phi)dt + \mathcal{V}_f (\dot \Phi))dt +(0,\comp(\cdot,\Phi))\bullet dW,
\end{align}
with $B$ the geodesic spray on $\Dmu$. We have used Stratonovich form here according to the ``transfer principle'' (see \cite{Bis1981}, \cite{Eme1990} and references therein), and ultimately because the Stratonovich chain rule has the same form of the classical chain rule, hence the formal computations in the deterministic case go through also in the stochastic case.

In the next subsection we show the first main result, that is existence and uniqueness for the stochastic Euler equation in Lagrangian form \eqref{Euler:stochastic_Lagrangian}. In the subsequent subsection we prove rigorously the link between the Eulerian and the Lagrangian form and derive the second main result, that is existence and uniqueness for the stochastic Euler equation in Eulerian form \eqref{Euler:stochastic}.

\subsection*{Local well-posedness in Lagrangian formulation}\addcontentsline{toc}{subsection}{Local well-posedness in Lagrangian formulation}

The objective of this section is to combine Theorem~\ref{thm:local_abstract_SDE_result} for SDE on Hilbert manifolds with Proposition~\ref{prop:VFisgood} for regularity of right translated vector fields to obtain local well-posedness for the stochastic Euler equation in Lagrangian form \eqref{Euler:stochastic_Lagrangian}. To deal with the diffusion term, we take a similar approach to \cite[Chapter VIII]{Elw82} and \cite{BrzElw1996}, which however are concerned with stochastic flows for finite-dimensional SDEs.

Now we define rigorously the drift and diffusion in \eqref{Euler:stochastic_Lagrangian}. We recall that $K$ is a compact (oriented) manifold $K$ (possibly with boundary), of dimension $d$, together with a Riemannian metric and its associated volume form $\mu$. We take $M=\Dmu$ with $s>d/2+1$. We assume to have a probability space $(\Omega,\cA,P)$ and a filtration $(\cF_t)_t$ as in \ref{setup:filtration_standard}.

The drift $B\colon T\Dmu \rightarrow T^2\Dmu$ in \eqref{Euler:stochastic_Lagrangian} is the geodesic spray associated to the right invariant $L^2$-metric on $\Dmu$.\footnote{Recall that the geodesic spray is the unique spray $F$ associated to a Riemannian metric such that a $C^2$-curve $\alpha$ is a geodesic if and only if $\frac{\dd^2}{\dd t^2} \alpha = F(\tfrac{\dd}{\dd t} \alpha)$, cf.\ \cite[IV \S 3 and VII \S 7]{MR1666820}} The smoothness of the drift is the main result by Ebin and Marsden:

\begin{setup}[{\cite[Theorem 11.2]{EM70}}]\label{thm:em70}
If $s>d/2+1$ then the geodesic spray $B$ on $\Diff_\mu^s(K)$ corresponding to the deterministic Euler equations is $C^\infty$. Furthermore, let $f \colon \R \rightarrow \mathfrak{X}_{\mu}^{s+1} (K)$ be continuous and consider $B_f := B + \mathcal{V}_f$ (with $\mathcal{V}_f$ as in \ref{setup:Eulerforce}). The second order vector field $B_f$ is associated to the Euler equation with forces is continuous and for every fixed $t$, $B_f(t,\cdot)$ is of class $C^1$. 
In particular, both $B$ and $B_f$ are of class $C^\infty$ (resp.\ $B_f(t,\cdot)$ is for $t$ fixed of class $C^{0,1}_{loc}$) in every chart.
\end{setup}

\begin{hp}\label{hp:noise}
Given $s^\prime$ non-negative integer, the process $W$ is a $Q$-Brownian motion (with respect to $(\cF_t)_t$) on $\mathfrak{X}^{s^\prime}_\mu(K)$, for some symmetric, positive-semidefinite and trace-class operator $Q$ on $\mathfrak{X}^{s^\prime}_\mu(K)$.
\end{hp}

For $s^\prime\ge s+2$ (without loss of generality, $s^\prime\ge s+2$), the diffusion coefficient in \eqref{Euler:stochastic_Lagrangian} is then $\Sigma\colon T\Dmu\rightarrow L(\mathfrak{X}^{s^\prime}_\mu(K);T^2\Dmu)$, where $\Sigma(\eta)V = E_V(\eta)$ for $\eta$ in $T\Dmu$ and $V$ in $\mathfrak{X}^{s^\prime}_\mu(K)$. We recall that
\begin{align*}
E_V(\eta) = \text{vl}_{T\Dmu} (\eta,\comp(V,\pi(\eta))).
\end{align*}

We assume that the initial datum $\eta_0$ satisfies $\pi(\eta_0)=\id$. The equation \eqref{Euler:stochastic_Lagrangian} now makes sense as SDE on the manifold $T\Dmu$, as in Definition \ref{def:SDE_man_sol} (as we will see in the proof of \ref{thm:main_Lagrangian}, the diffusion coefficient is $C^1$ in a neighbourhood of $\eta_0$).

We can now formulate our local well-posedness result:

\begin{thm}\label{thm:main_Lagrangian}
Fix $s>d/2+1$ and suppose Assumption \ref{hp:noise} with $s^\prime= s+2$, take $\eta_0$ in $\mathfrak{X}^s_\mu(K)$ with $\pi(\eta_0)=\id$ (the identity map). Then local strong well-posedness, in the sense of Theorem~\ref{thm:local_abstract_SDE_result}, holds for the Lagrangian formulation \eqref{Euler:stochastic_Lagrangian} of the stochastic Euler equations (with or without forces) on $T\Dmu$.
\end{thm}

\begin{proof}
We want to apply Theorem~\ref{thm:local_abstract_SDE_result}. We take a chart $\psi$ as in \ref{setup:VLIFT}; restricting the domain $O$ of $\psi$, we can assume that $O$ is an open bounded set and $\psi$ is bounded with its $100$ derivatives. Then $B^\psi$ is smooth, in particular $C^{0,1}$ (in the presence of a forcing term $f$ as in \ref{thm:em70} we consider instead $B_f^\psi$ which is of class $C^{0,1}$), on $\psi(O)$. Moreover $\Sigma^\psi = (e^{\psi,2})^\wedge$ is $C^{1,1}$ on $\psi(O)$ by Proposition \ref{prop:VFisgood}. Hence the result follows from Theorem~\ref{thm:local_abstract_SDE_result}.
\end{proof}

\subsection*{Link between Eulerian and Lagrangian formulations}\addcontentsline{toc}{subsection}{Link between Eulerian and Lagrangian formulations}

In this subsection we come back to the stochastic Euler equation in Eulerian form \eqref{Euler:Lagrangian}: we prove the rigorous link between the Eulerian form and the Lagrangian form and we derive the well-posedness result for the Eulerian form. 
Here we will establish the results first only for the Euler equation without deterministic forcing term $f$. This keeps the formulae simpler. However, we stress that the same results hold also in the presence of a forcing term $f$ (by taking a trivial modification of the proof), cf.\ Remark \ref{rem:forcingallowed} below.  
Let $\Pi$ denote the Leray projector on the divergence-free vector fields (cf.\ \cite[Section 5]{MR1849348} and \cite[Appendix A]{EM70}). On the noise, we take Assumption \ref{hp:noise} with $s^\prime\ge s-1$.

\begin{defn}
A local (strong and smooth) solution to the stochastic Euler equation in Eulerian form \eqref{Euler:stochastic} is an $\mathfrak{X}^s_\mu(K)$-valued $\cF$-progressively measurable process $u=(u(t))_{[0,\tau)}$, with $\tau$ accessible stopping time and $\tau>0$ $P$-a.s., with $P$-a.s. continuous paths in $\mathfrak{X}^s_\mu(K)$, such that
\begin{align}
u(t) = u(0) -\int_0^t \Pi[\nabla_{u(r)} u(r)] \, dr + W(t), \quad \forall t\in[0,\tau).\label{eq:Euler_def}
\end{align}
\end{defn}

\begin{rem}
For $s>d/2+1$, if $t\mapsto u(t)$ is a continuous path with values in $\mathfrak{X}^s_\mu(K)$, then, by Sobolev embedding, $t\mapsto u(t)$ is a continuous path with values in $C^1(K)$ (since $K$ is compact, $C^1(K)$ is a Banach space, cf.\ e.g.\ \cite[p.24]{MR0248880}). Hence $t\mapsto \nabla_{u(t)} u(t)$ and so $t\mapsto \Pi[\nabla_{u(t)} u(t)]$ are continuous paths in $H^{s-1}(K)$, in particular the integral of $\Pi[\nabla_{u(t)} u(t)]$ makes sense. Hence the equality \eqref{eq:Euler_def} holds in $H^{s-1}$.
\end{rem}

The link between the Eulerian and the Lagrangian form is proved in the following:

\begin{thm}\label{thm:equivalence_flow_pde}
The Eulerian form \eqref{Euler:stochastic} and the Lagrangian form \eqref{Euler:stochastic_Lagrangian} of the stochastic Euler equation are equivalent, in the following sense:
\begin{itemize}
\item Fix $s>d/2+3$ and assume that $W$ satisfies \ref{hp:noise} with $s^\prime= s+2$. If $\eta$ is a solution on $[0,\tau)$ to the Lagrangian form in $T\Diff^s_\mu(K)$, with $\pi(\eta_0)=\id$, then $u(t)=\eta(t)\circ \pi(\eta(t))^{-1}$ is a solution on $[0,\tau)$ to the Eulerian form in $\mathfrak{X}^s_\mu$.
\item Conversely, fix $s>d/2+4$ and assume that $W$ satisfies \ref{hp:noise} with $s^\prime=s-1$. If $u$ is a solution on $[0,\tau)$ to the Eulerian form in $\mathfrak{X}^s_\mu$ and $\Phi$ is the unique flow solution on $[0,\tau')$ (for some accessible $\tau'\le \tau$) to the (random) ODE
\begin{align*}
d\Phi(t) = u(t,\Phi(t)) dt,\quad \Phi(0)=\id,
\end{align*}
then $u(t)\circ \Phi(t)$ is a solution on $[0,\tau')$ to the Lagrangian form on $T\Diff^{s-3}_\mu(K)$.
\end{itemize}
\end{thm}

Before we establish the theorem, we need establish differentiability and identities for certain tangent mappings first. These will be used in the proof.

\begin{lem}\label{lem:aux:deriv}
The following mappings are of class $C^2$
\begin{align*}
 G \colon \VFsmu[s-1] \times \Dmu[s-3] &\rightarrow T\Dmu[s-3] \subseteq H^{s-3}(K,TK),\quad (V,\Phi) \mapsto V\circ \Phi,\\
 F\colon T\Dmu[s] &\rightarrow \VFsmu[s-2] , V\circ \Phi \mapsto V\circ \Phi \circ \Phi^{-1} =  V.
\end{align*}
their tangent mappings are (up to canonical identification) given by:
\begin{align}
TG ((X,Y),V \circ \Phi) = TX\circ V \circ \Phi + Y_X \circ \Phi \in H^{s-3} (K,T^2K) \label{eq:TG}\\ 
T_{V\circ \Phi}F (\zeta) = T_{(V\circ \Phi, \Phi^{-1})} \mathrm{comp} (\zeta, T _\Phi\mathrm{inv} \circ T\pi (\zeta)) = (\zeta - TV \circ T\pi (\zeta))\circ \Phi^{-1}.
\label{eq:TF}
\end{align}
Here $Y_X \colon K \rightarrow T^2K$ is the vertical field locally (in a chart $\kappa$) conjugate to $k \mapsto (k,X^\kappa(k),0,Y^\kappa(k))$. 
\end{lem}

\begin{proof}
Recall from Appendix \ref{app: Sobolev}, \eqref{tangent:ident} that we can identify the tangent manifold of the manifold of Sobolev mappings as
\begin{align*}
T H^{s} (K,K') = H^s (K,TK') \text{ for all } s \geq \text{dim } K/2 +1.
\end{align*}
As $\Dmu[s-3] \subseteq H^{s-3} (K,K)$ is a submanifold, we use the inclusion to identify the tangent bundle of $\Dmu[s-3]$ as a subbundle of $TH^{s-3}(K,K)$.

By  restricting the composition $H^{s-1}(K,TK) \times \Dmu[s-3] \rightarrow H^{s-3}(K,TK)$ to a submanifold, we obtain $G$. Hence $G$  is of class $C^2$ as a map to $T\Dmu[s-3]$ (which we have identified as $TH^{s-3}(K,TK)|_{\Dmu[s-3]}$) by \ref{setup: HL} and \ref{setup:boundary}. It is well known that the tangent map of the composition is the sum of the tangent maps of left and right composition, \ref{setup: HL} and see e.g.\ \cite[Corollary 11.6]{MR583436}. We will use this to compute an explicit formula for
$$TG \colon T\VFsmu[s-1] \times T\Dmu[s-3] \rightarrow TH^{s-3}(K,TK) = H^{s-3} (K,T^2K).$$ As $\VFsmu[s-1] \subseteq H^{s-1}(K,TK)$ is a Hilbert space, $T\VFsmu[s-1] = \VFsmu[s-1]\times \VFsmu[s-1]$ holds. To combine this with the formula for the left and right composition \ref{setup: HL}, we need to identify $T\VFsmu[s-1]$ as a subspace of $TH^{s-1}(K,TK) = H^{s-1}(K,T^2K)$. By definition, any element $X$ in $\VFsmu[s-1]$ satisfies $\pi_K \circ X = \id_K$, for $\pi_K \colon TK \rightarrow K$ the bundle projection. Hence the differential of the embedding $\VFsmu[s-1] \subseteq H^{s-1}(K,TK)$ identifies $\VFsmu[s-1]^2 \rightarrow H^{s-1} (K,T^2K), (X,Y) \mapsto Y_X$ via the vertical lift \ref{setup:VLIFT}, i.e. in charts $T^2\kappa,\kappa$ ($\kappa$ being any generic chart of $K$), $Y_X$ is conjugate to a mapping $k \mapsto (k,X^\kappa(k),0,Y^\kappa(k))$, where $X^\kappa = \text{pr}_2 \circ T\kappa \circ X \circ \kappa^{-1}$ is the principal part of $X$. Consider $V\circ \Phi \in T_\Phi \Dmu[s-3]$ with $V\in \VFsmu[s-3]$. Then the identification together with the formulae for left and right composition \ref{setup: HL} yields \eqref{eq:TG}.

We now turn to the mapping $ F\colon T\Dmu[s] \rightarrow \VFsmu[s-2] , V\circ \Phi \mapsto V\circ \Phi \circ \Phi^{-1} =  V.$
Note that $F(V\circ \Phi) =\text{comp} (V \circ \Phi , \text{inv} (\pi_{T\Dmu[s]} (V\circ \Phi))$, where $\pi_{T\Dmu[s]}$ is the bundle projection, $\text{comp} \colon H^{s} (K,TK) \times \Dmu[s-3] \rightarrow H^{s-2}(K,TK)$ the composition map and $\text{inv} \colon \Dmu[s] \rightarrow \Dmu[s-2]$ the inversion map. As all of these mappings are at least of class $C^2$ by \ref{setup: HL}, we see that $F$ is $C^2$ as a map to the closed subvectorspace $\VFsmu[s-2] \subseteq H^{s-2}(K,TK)$. Now we wish to leverage the formulae for the tangent mappings of composition and inversion in \ref{setup: HL} to obtain a formula for the tangent of $F$. To this end, we identify $\VFsmu[s-1] \subseteq H^{s}(K,TK)$ and $\VFsmu[s]\times \VFsmu[s] = T\VFsmu[s]$ with mappings taking their image in the vertical part of the double tangent bundle. Identify $T^2\Dmu[s]\subseteq H^{s}(K,T^2K)$ and apply the chain rule. Then the derivative of $F$ viewed as an element in $H^{s}(K,T^2K)$ becomes for any $\zeta$ in $T^2\Dmu[s]\subseteq H^{s}(K,T^2K)$,
\eqref{eq:TF}
The symbol '$-$' in \eqref{eq:TF} means the operation in $T_{V\circ \Phi} T\Dmu[s-1]$ and we observe that since $V$ is a vector field, \eqref{eq:TF} indeed takes its values in the vertical part of $T^2K$.
\end{proof}

\begin{proof}[Proof of Theorem \ref{thm:equivalence_flow_pde}]
We start with the second statement. We are given $u$ solution to the stochastic Euler equations on $[0,\tau)$. We start constructing the Euler flow $\Phi$: we show that there exists a unique progressively measurable process $\Phi\colon [0,\tau^\prime)\times\Omega \rightarrow \Dmu[s-1]$, where $\tau^\prime$ is a suitable positive accessible stopping time, such that, $P$-a.s.,
\begin{align}
\dot{\Phi}(t) = u(t,\Phi(t)),\quad \forall t \in[0,\tau^\prime).\label{eq:ODE_position}
\end{align}
For this, we pick a chart $\psi_\alpha$ of $\Dmu[s-1]$, with domain $O_\alpha$, such that $\psi_\alpha (O_\alpha)$ is a bounded zero neighbourhood with $\psi_{\alpha} (\id) = 0$. Then the mapping 
\begin{align*}
\xi \colon \VFsmu[s]\times \psi_\alpha (O_\alpha) \ni (V,\Phi)\mapsto T\psi_\alpha \circ \text{comp}(V,\psi_\alpha^{-1} (\Phi)) \in TO \subseteq (\VFsmu[s-1])^2
\end{align*}
is $C^1$. We observe that for every fixed $u_0 \in \VFsmu[s]$ the norm $\lVert D\xi \rVert_{\text{op}}$ is bounded on the compact set $C \coloneq \{tu_0 \mid t \in [0,1]\} \times \{0\}$. Thus we find an open convex neighbourhood $U$ of $C$ such that the norm of $D\xi$ is bounded, whence $\xi$ is of class $C^{0,1}$ (with uniform Lipschitz constant) on $U$. Applying now the Wallace Lemma \cite[3.2.10]{Eng89}, we may shrink $U$ to a neighbourhood $U^\prime \times W$ of $C$ such that, for every $V$ in $U^\prime$, $\xi (V,\cdot)$ is $C^{0,1}$ on $W$ and the Lipschitz constant is uniformly bounded in $V \in U^\prime$. As a consequence, calling $\tau^\prime$ the minimum between $\tau$ and the (accessible) exit time of $u$ from $U^\prime$, the random drift
the random drift
\begin{align*}
[0,\tau^\prime)\times \Omega \times \Dmu[s-1]\ni (t,\omega,\Phi) \rightarrow \comp(u(t,\omega),\Phi) \in T\Dmu[s-1]
\end{align*}
is, in the chart $\psi_\alpha$, $P$-a.s. Lipschitz in $\Phi$, uniformly in $t$, and also continuous in $(t,\Phi)$, because $\xi$ is continuous and the path $t\mapsto u(t,\omega)$ is continuous in $\VFsmu[s]$. Therefore we can apply Theorem \ref{thm:local_abstract_SDE_result} and Remark \ref{rmk:man_randm_drift} with this random drift and with zero diffusion coefficient: possibly taking a smaller $\tau'$, we get existence and uniqueness on $[0,\tau^\prime)$ of a progressively measurable solution $\Phi$ to \eqref{eq:ODE_position}.

As a consequence, the process $(u,\Phi)$ satisfies the Stratonovich differential, on the manifold $\VFsmu[s-1] \times \Dmu[s-3]$,
\begin{align*}
d(u,\Phi) = (-\Pi[\nabla_{u} u],u\circ \Phi) dt +(I,0) \bullet dW,
\end{align*}
where $I\colon \VFsmu[s^\prime]\rightarrow \VFsmu[s-1]$ is the inclusion map. Note that the drift can be extended to a Borel function on $\VFsmu[s-1] \times \Dmu[s-3]$ (for example setting the drift equal to $(0,u\circ \Phi)$ for $u$ not in $\VFsmu[s]$), which fits into Remark \ref{rem:invariance_SDE_map}.

In view of It\^o formula, we need an expression for the derivative of the composition map $G(V,\Phi) = V\circ \Phi$ from Lemma \ref{lem:aux:deriv}. 
To give an explicit formula, we evaluate \eqref{eq:TG} in $k\in K$ and localise in a manifold chart $\kappa$ for $K$ around $\Phi(k)$. Writing $k_\star = \kappa (\Phi(k))$, \eqref{eq:TG} shows that $\text{ev}_k (TG((X,Y),V \circ \Phi))$ equals
\begin{equation}\label{eq:localstar}
  T^2\kappa^{-1} (k_\star, X^\kappa (k_\star), V^\kappa (k_\star), DX^\kappa (k_\star , V^\kappa (k_\star)) + Y^\kappa (k_\star)).
 \end{equation}
Apply the It\^o formula to see that the process $\eta(t)=u(t)\circ \Phi(t)$ satisfies
\begin{align*}
d\eta(t) = TG((u(t),-\Pi[\nabla_{u(t)} u(t)]),u(t)\circ \Phi) dt + TG((u(t),I),0\circ \Phi) \bullet dW_t,
\end{align*}
where $TG((V,I),0\circ \Phi)$ is the linear mapping in $L(\VFsmu[s^\prime],T_{V\circ \Phi}\Dmu[s-3])$ defined by $TG((V,I),0\circ \Phi) U =TG((V,U),0\circ \Phi)$ for every $U$ in $\VFsmu[s^\prime]$. In order to identify $TG ((V,-\Pi[\nabla_V V]),V\circ \Phi)$ with the drift $B(V\circ \Phi)$ in the Lagrangian form, we recall first from \cite[Proposition 14.2]{EM70} the formula for the spray of the geodesic equation evaluated at $\eta \in T_\Phi \Dmu[s-3]$ reads
\begin{align} \label{eq:geodesic_spray_EM}
B(\eta)=T(\eta\circ \Phi^{-1})\circ \eta -\text{vl}_{T\Dmu[s-3]}(\Pi (\nabla_{\eta\circ \Phi^{-1}} \eta\circ \Phi^{-1})) \circ \Phi.
\end{align}
Now we evaluate the drift $TG ((V,-\Pi[\nabla_V V]),V\circ \Phi)$ in $k$, and localize via \eqref{eq:localstar} in the chart $\kappa$: for every $V \in \VFsmu[s], \Phi \in \Dmu[s-3]$ and $k\in K$:
\begin{align*}
T^2\kappa^{-1} (k_\star,V^\kappa (k_\star) ,V^\kappa (k_\star),(DV \circ V)^\kappa (k_\star)  -(\Pi[\nabla_V V])^\kappa (k_\star)) &= B(V\circ \Phi)(k)
\end{align*}
As $k$ was arbitrary, we obtain the desired identity of the drifts for $V$ in $\VFsmu[s]$ (which is enough, since $u$ lives in $\VFsmu[s]$ $P$-a.s.). For the diffusion coefficient, we argue similarly to obtain for every $V \in \VFsmu[s-1]$ and $\Phi \in \Dmu[s-1]$, for every $U$ in $\VFsmu[s^\prime]$, for every $k\in K$,
\begin{align*}
\text{ev}_k (TG((V,U),0\circ \Phi)) = T^2 \kappa^{-1}(k_\star ,V^\kappa (k_\star) ,0,U^\kappa (k_\star)) = [\Sigma(V\circ \Phi)U](k),
\end{align*}
therefore also the diffusion coefficient $TG((V,I),0\circ \Phi)$ and $\Sigma(V\circ\Phi)$ coincide. Hence we apply Lemma \ref{lem:invariance_SDE_diffeo} and Remark \ref{rem:invariance_SDE_map} to obtain the Eulerian form. The proof of the second statement is complete.

Let us now consider the first statement: We are given a solution $\eta$ to the Eulerian form on $[0,\tau)$, with $\pi(\eta_0)=\id$, and use the auxiliary map 
$F(V\circ \Phi) := V\circ \Phi \circ \Phi^{-1} =  V$ from Lemma \ref{lem:aux:deriv}.
To turn \eqref{eq:TF} into an explicit formula for the derivative of $F$ as a mapping to $\VFsmu[s]\times \VFsmu[s]$, we need to localise $\zeta \in H^{s}(K,T^2K)$ in a chart. Thus let $\kappa$ be a chart around $k$ and take the chart representation of $\zeta \circ \Phi^{-1}$ as $Z^\kappa(\kappa(k)) = (\kappa(k) , V^\kappa , Z^{1,\kappa}  , Z^{2,\kappa})$ (where we suppress the argument $\kappa(k)$ for readability), where we used the identification $T^2 O_\kappa \cong (O_\kappa \times \R^d) \times (\R^d \times \R^d)$. This allows us to obtain a local formula for $DF$ as a mapping to $\VFsmu[s]\times \VFsmu[s]$ (the image viewed locally over $O_\kappa$):%
\begin{align}\label{Ftangent:localised}
D_{V \circ \Phi } F  (\zeta)(k) = 
(V^\kappa(\kappa(k))) ,Z^{2,\kappa}(\kappa(k)) -DV^\kappa\circ Z^{1,\kappa} (\kappa(k))),
\end{align}
for arbitrary $k$ in $K$.
Since $F$ is $C^2$ by Lemma \ref{lem:aux:deriv}, we can apply It\^o formula to $u=F(\eta)$ and get
\begin{align*}
du = DF \circ B(u\circ \Phi) dt + DF \circ \Sigma(u\circ \Phi) \bullet dW.
\end{align*}
To identify the drift $DF \circ B(V\circ \Phi)$ with the Eulerian drift $(V,-\Pi[\nabla_V V])$, we recall again formula \eqref{eq:geodesic_spray_EM} from \cite[Proposition 14.2]{EM70} and we localize $B(V\circ \Phi) \circ \Phi^{-1}$ in a chart $\kappa$ around $k$ (suppressing $\kappa(k)$ again):
\begin{align*}
[B(V\circ \Phi) \circ \Phi^{-1}]^\kappa (\kappa(k)) = (\kappa(k),V^\kappa,V^\kappa,(DV^\kappa)\circ V^\kappa-(\Pi[\nabla_V V])^\kappa).
\end{align*}
Now evaluate $TF \circ B(V\circ \Phi)$ in $k\in K$ using \eqref{Ftangent:localised}. The principal part becomes for every $V\circ \Phi \in T\Dmu[s]$ and $k$,
\begin{align}\label{eq:TFB}
DF \circ B(V\circ \Phi) (k) &= (V^\kappa,-\Pi[\nabla_V V]^\kappa)
\end{align}
where we have suppressed the argument $\kappa (k)$ on the right hand side. Now \eqref{eq:TFB} yields the desired identity for the drift. Likewise, for the diffusion coefficient, \eqref{Ftangent:localised} reduces for every $V\circ \Phi$, for every $U$ in $\VFsmu[s^\prime]$, for every $k\in K$, to
\begin{align*}
DF ( \Sigma(V\circ \Phi),U) (k) = (V^\kappa (\kappa (k)),U^\kappa(\kappa (k))),
\end{align*}
hence $DF (\Sigma(V\circ \Phi))$ coincide with the Eulerian diffusion coefficient $(V,I)$. We can now apply Lemma \ref{lem:invariance_SDE_diffeo} and Remark \ref{rem:invariance_SDE_map} to obtain the Eulerian form. The equation holds a priori on $\VFsmu[s-2]$, but $u$ and all the integrands are continuous paths in $\VFsmu[s-2]$, hence the equation holds on $\VFsmu[s-1]$ as well. The proof is complete.
\end{proof}

\begin{rem}
	Theorem~\ref{thm:equivalence_flow_pde} above can be seen as a stochastic analogue of \cite[Thm~14.4]{EM70}, where it is proven that the deterministic incompressible Euler equation with smooth forcing on $\mathfrak{X}_\mu^s(K)$ is equivalent to the deterministic Euler flow on $\Diff^s_\mu(K)$ if $s>d/2+2$.
	Notice, however, that in our Theorem~\ref{thm:equivalence_flow_pde} we do not obtain equivalence: solutions to the stochastic Lagrangian form give solutions to the stochastic Eulerian form only for $s>d/2+3$, and the other direction requires even $s>d/2+4$.
	This comes mainly from the $C^2$ regularity necessary in the stochastic analysis to apply the It\^o formula.
\end{rem}

\begin{cor}\label{cor:locwell}
Fix $s>d/2+4$ and assume that $W$ satisfies Assumption \ref{hp:noise} with $s^\prime =s+2$. Then, for every $u_0$ in $\VFsmu$, local strong well-posedness, in the sense of Theorem~\ref{thm:local_abstract_SDE_result}, holds for the stochastic Euler equation in Eulerian form \eqref{Euler:stochastic} among $\VFsmu$-valued solutions.
\end{cor}

\begin{proof}
The result follows from the well-posedness theorem \ref{thm:main_Lagrangian} for the Lagrangian form and the link between the Eulerian and Lagrangian viewpoints Theorem \ref{thm:equivalence_flow_pde}.
\end{proof}

\begin{rem}\label{rem:forcingallowed}
The proof of Theorem \ref{thm:equivalence_flow_pde} used certain algebraic identities of the spray $B$ and its continuity. 
In the presence of an additional deterministic forcing term $f$, one has to replace the spray $B$ by the second order vector field $B_f$ discussed in \ref{setup:Eulerforce} and Theorem \ref{thm:main_Lagrangian}. 
Due to the work of Ebin and Marsden \cite[Section 11]{EM70}, similar algebraic identifies as the ones used for $B$ in the proof of Theorem \ref{thm:equivalence_flow_pde} hold for $B_f$. In addition, we have seen that for continuous $f\colon \R \rightarrow \VFsmu[s+1]$, also the second order vector field $B_f$ is locally $C^{0,1}$. Thus, the results obtained in Theorem \ref{thm:equivalence_flow_pde} and Corollary \ref{cor:locwell} hold (by the same proof) also for solutions of the stochastic Euler equation in the presence of an additional deterministic forcing term.
\end{rem}

\section{Extensions of the local well-posedness results}\label{sect:above_beyond}

In this section we discuss various natural extensions of the results of the present paper. In order to get a simple introductory presentation, we chose in many cases not to give the most general version possible with the methods developed in the present work. These topics are now discussed.

\subsection*{Sobolev spaces of different regularity}\addcontentsline{toc}{subsection}{Sobolev spaces of different regularity}

Throughout the main text our assumption was that the Sobolev spaces $H^s (K,M)$ and $\Diff^s (K)$ are of integer order $s$ and the derivatives are contained in $L^2$.
There are two natural ways to modify these requirements:
\begin{enumerate}
\item Pass to fractional order Sobolev spaces.
\item Replace derivatives in $L^2$ by derivatives in some $L^p$ for $1<p<\infty$.
\end{enumerate}
Reviewing the arguments used, none of them are tied to integer order Sobolev spaces\footnote{The exception being the description of the inner products giving the Hilbert space structure of the Sobolev mappings and the right-invariant strong Riemannian metric on $\Diff^s (K)$, cf.\ \ref{setup: volpres}. However, these spaces always admit a right-invariant strong Riemannian metric and the explicit form is irrelevant to our arguments.} whence they carry over verbatim to fractional order Sobolev spaces.
The main issue why we refrained from using fractional order Sobolev spaces is due to the fact that the underlying manifold $K$ is allowed to have a smooth boundary.
Hence our discussion has to distinguish two cases:

\begin{setup}Fractional order Sobolev spaces for $K$ without boundary:
 If $K$ has no boundary, it is no problem to define fractional order Sobolev spaces on manifolds via the approach outlined in \cite{IKT13} (note that loc.cit.\ only defines integer order Sobolev spaces on compact manifolds, but the approach carries over to fractional order, as noted in \cite[Section 5]{MR3635359}).
 Indeed one then retains all the necessary tools (e.g.\ Sobolev embedding theorems, differentiability of the composition, etc.) to carry out the constructions leading to Theorems \ref{thm:main_Lagrangian} and \ref{thm:equivalence_flow_pde}.\end{setup}
 Hence we can state the following:
  
 \begin{prop} 
 If $K$ is a compact smooth manifold without boundary. Then Theorems \ref{thm:main_Lagrangian} and \ref{thm:equivalence_flow_pde} remain valid for fractional order Sobolev index $s$ which satisfy the assumptions of these theorems. In particular, both main theorems from the introduction remain valid for fractional order Sobolev index $s$.
\end{prop}

 \begin{setup}Fractional order Sobolev spaces for $K$ with smooth boundary:
  If $K$ is allowed to have a smooth boundary, we still believe that the results stated in the boundary-less case carry over. However, to the best of our knowledge, the only source in the literature sketching the construction of a manifold of mappings for fractional Sobolev regularity on a manifold with smooth boundary is the sketch contained in \cite[Section 5]{Mic19}. We strongly believe that the cited results will allow a similar theory for equations on fractional order Sobolev spaces as in the boundary-less case. Working out the details is however beyond the scope of the present paper.
 \end{setup}
If we change the Sobolev spaces in requiring that the derivatives should be contained in $L^p$ (relaxing the requirement that $p=2$) one notices first that the spaces of Sobolev mappings are no longer (modelled on) Hilbert spaces. Instead we are in the weaker Banach space setting. It is well known that the classical Ebin-Marsden analysis can be carried out in the $L^p$ setting, \cite{BaB74}. However, also the stochastic setting outlined in Section~\ref{sect: SDE:HMfd} has to be modified since it relied on Hilbert space techniques, see \cite{BaE01} for how to do this. Though the authors think that this is possible in principle, we remark here that it would require a significant amount of work. This is not only due to the stochastic setting used. For example, we have used cut-off functions which exist on Hilbert, but not necessarily on Banach spaces \cite[Chapter III]{MR1471480}.

\subsection*{Maximal solutions and preservation of regularity}\addcontentsline{toc}{subsection}{Maximal solutions and preservation of regularity}

Our aim in the present paper was to establish a local existence and uniqueness theory for a stochastic version of the Euler equation. From the perspective of the infinite-dimensional manifold of Sobolev diffeomorphisms, all solutions to the Euler equation start at the identity element. Hence, to treat local existence and uniqueness we could restrict to a single fixed chart around the identity. As a result, the solutions we construct will (in general) not be maximal as we stop the solution once it leaves the domain of the manifold chart. 
While this is sufficient for local existence and uniqueness considerations, it might be desirable to establish the existence of maximal (in time) solutions to the stochastic differential equations. With some additional effort this is possible, see \cite[Chapter VII]{Elw82}, but again it requires some work (essentially because, even at fixed time, a solution $X(\omega)$ of an SDE can live in different chart domains for different $\omega$). In order to keep the presentation self-contained, we did not include maximality here. Moreover, some attention is needed to relate maximality at the Lagrangian and Eulerian level (to our knowledge, this point is not investigated even in the deterministic setting in \cite{EM70}). However we expect existence (and uniqueness) of a maximal solution to hold for both Lagrangian and Eulerian level.

Once maximality is established, one can study the problem of preservation of regularity. Indeed, it is well known that the deterministic Euler equation preserves regularity of initial conditions. Namely, one has the famous
\begin{setup}[{``no loss no gain'' theorem \cite[Theorem 12.1]{EM70}}]
	\emph{Let $K$ be a compact manifold (possibly with boundary), $s> \frac{\mathrm{dim }K}{2} +1$ and $\eta$ the solution of a second order equation on $\Diff_{\mu}^s(M)$ given by a smooth right invariant second order vector field (in particular, by the geodesic spray). If $(\eta (0), \eta'(0)) \in \Diff_\mu^{s+k} (K) \times T_{\eta(0)} \Diff_{\mu}^{s+k} (K)$ then $\eta(t)$ is $H^{s+k}$ on the interior of $K$ for all $t$ in its domain of definition. }
\end{setup}
In particular, the interval of existence of the solution does neither increase nor decrease based on regularity of the initial data (that is, there cannot exist a solution which is $H^{s+k}$ up to some time, and only $H^s$ after). While we expect that a version of the theorem also holds in our setting, the argument in loc.cit.\ cannot be readily adapted to stochastic differential equations: the main reason is that this argument uses the existence of a flow solution to the ODE on the \textit{infinite}-dimensional manifold $T\Diff_{\mu}(K)$, something which is not clear to hold in the stochastic case.

\subsection*{Optimality in the Sobolev index}\addcontentsline{toc}{subsection}{Optimality in the Sobolev index}

In the deterministic setting, \cite[Theorem 14.4]{EM70} obtains equivalence of the Eulerian and Lagrangian forms at least for $s>d/2+2$, one can probably get even just $s>d/2+1$ with some attention in the proof. We expect that also in the stochastic setting this threshold can be reached and that Corollary \ref{cor:locwell} holds with the threshold $s>d/2+1$. This can be done using some technical tools, like embedding of split manifolds (possibly with boundary) and regularisation arguments. For example, in the passage from Eulerian to Lagrangian viewpoint, one may regularize the Eulerian solution and apply It\^o formula, then send regularization parameter to $0$, and so one should get an equation for position and velocity of each particle (that is, for $(\Phi(k),u(\Phi(k)))$ for each fixed $k$). Finally, using the embedding of $T\Diff^s_\mu$ as split manifold into the Hilbert space of $H^s$ mappings, one may show the equation for position and velocity $(\Phi,u(\Phi))$ as an element in $T\Diff^s_\mu(K)$. However the use of these technical tools requires a longer and careful analysis, which is outside the scope of this paper, hence we will address this result in a subsequent paper.


\subsection*{Multiplicative noise}\addcontentsline{toc}{subsection}{Multiplicative noise}

In this paper, we considered additive noise as the simplest example of noise. We expect that the framework can be extended to other kinds of noise, for example a multiplicative noise of Nemytskii type, namely $g(x,u(t,x)) \bullet dW(t,x)$ for $g:TK \to \R$ sufficiently regular: in this case, the noise in the Lagrangian formulation should be of the form $\Sigma(\eta) \bullet dW = \text{vl}_{T\Dmu} \comp(g\bullet dW,\eta)$, which has similar regularity properties to the additive noise $\text{vl}_{T\Dmu} \comp(\bullet dW,\pi(\eta))$. We leave the precise analysis of Euler equations with this and other noises for future works.

\subsection*{Extending the mechanism to other Euler-Arnold equations}\addcontentsline{toc}{subsection}{Extending the mechanism to other Euler-Arnold equations}
The Euler equation is the prototypical example of a PDE which can be rewritten as an ODE on an infinite-dimensional configuration space. 
However, the same is true for a large class of PDEs: the so-called \emph{Euler-Arnold equations}. 
We refer to \cite[Example 4.18]{MR2456522} for a list of examples and references to the literature, including the Camassa-Holm and Hunter-Saxton equations.
For many of these examples one can adapt the strategy of Ebin and Marsden to obtain local well-posedness.
Our stochastic framework developed in this paper is ignorant of the deterministic (drift) part of the equation, as long as it is a (locally) Lipschitz continuous vector field on the underlying Banach manifold.
Therefore, it can be applied, with small modification, to other stochastically forced Euler--Arnold equations.

To demonstrate the flexibility of the approach we give here a new result for a stochastic version of the \emph{averaged Euler equations}~\cite{Sh1998,MR1849348,MaRaSh2000,MaSh2003}.
This is a fluid model where nonlinear interactions for length-scales smaller than $\alpha>0$ are neglected.
We consider a stochastically forced version of these equations
\begin{equation}\label{eq:averaged_euler_stochastic}
\left\{
\begin{aligned}
& \frac{\partial m}{\partial t} + \nabla_u m + \alpha^2 \nabla u^\top \Delta u + \nabla p = \dot W \\
& m = (\id-\alpha^2\Delta)u \\
& \operatorname{div} u = 0 , \quad u(0,\cdot) = u_0 \, .
\end{aligned}
\right.
\end{equation}
Here, $\Delta$ denotes the Laplace-de Rham operator and the new variable $m$ is interpreted as momentum.
For simplicity, let us assume that the underlying compact manifold $K$ is without boundary. In the following, we take $s>d/2+1$ and assume Assumption \ref{hp:noise} on the noise with $s'\ge s-3$. The precise definition of solution is the following one:

\begin{defn}
A local (strong and smooth) solution to the stochastic averaged Euler equation in Eulerian form \eqref{eq:averaged_euler_stochastic} is an $\mathfrak{X}^s_\mu(K)$-valued $\cF$-progressively measurable process $u=(u(t))_{[0,\tau)}$, with $\tau$ accessible stopping time and $\tau>0$ $P$-a.s., with $P$-a.s. continuous paths in $\mathfrak{X}^s_\mu(K)$, such that for all $ t\in[0,\tau)$
\begin{align*}
&u(t) - u(0) \notag  \\ 
=&\int_0^t \Pi(\id-\alpha^2 \Delta)^{-1}[\nabla_{u} (\id-\alpha^2 \Delta)u+\alpha^2 \nabla u^\top \Delta u] \, dr + (\id-\alpha^2 \Delta)^{-1} W(t).
\end{align*}
\end{defn}
As for the Euler equations, the equations \eqref{eq:averaged_euler_stochastic} can be cast instead as a second order SDE in the Lagrangian variable $\Phi\in \Diff^s_\mu(K)$ via $\dot\Phi = u\circ\Phi$. Informally, in Lagrangian coordinates the equations take the form
\begin{align*}
\ddot\Phi = B_\alpha(\Phi, \dot\Phi) + ((\id-\alpha^2\Delta)^{-1}\dot W)\circ\Phi.
\end{align*}
Rigorously, the Lagrangian formulation is the following SDE on $T\Diff^s_\mu(K)$:
\begin{align}\label{eq:averaged_euler_stochastic_lag}
d\eta = B_\alpha(\eta)dt +\Sigma_\alpha(\eta)\circ dW,
\end{align}
where $B_\alpha:T\Diff_\mu^s(K) \to T^2\Diff_\mu^s(K)$ is a smooth bundle map~\cite[Thm.~3.3]{Sh1998}\cite{MR1849348} and $\Sigma_\alpha$ is defined by
\begin{align*}
& \Sigma_\alpha:T\Diff^s_\mu(K) \to L(\mathfrak{X}^s_\mu(K),T^2\Diff^s_\mu(K)), \\ & \Sigma_\alpha(\eta)V = \text{vl}_{T\Diff^s_\mu(K)}(\eta,\text{comp}((\id-\alpha^2\Delta)^{-1}V)).
\end{align*}
Applying our stochastic Ebin and Marsden framework then yields the following results.

\begin{thm}\label{thm:averaged_lag}
	Fix $s>d/2+1$ and suppose Assumption \ref{hp:noise} with $s'\ge s$.
	Then local strong well-posedness (in the sense of Theorem~\ref{thm:local_abstract_SDE_result}) holds for the Lagrangian formulation \eqref{eq:averaged_euler_stochastic_lag} on $T\Diff_\mu^s(K)$.
\end{thm}

Note that the regularity assumption on the noise is lower than what we get for the stochastic Euler equations.
This is because the relation between the momentum $m$ and the velocity field $u$ has a regularizing effect (hence \emph{averaged} Euler equations).
We also obtain a corresponding result in Eulerian variables:

\begin{thm}\label{thm:averaged_eul}
	Fix $s>d/2+4$ and suppose Assumption \ref{hp:noise} with $s'\ge s$.
	Then local strong well-posedness (in the sense of Theorem~\ref{thm:local_abstract_SDE_result}) holds for the stochastic averaged Euler equation~\eqref{eq:averaged_euler_stochastic} on $\mathfrak{X}^s_\mu(K)$.
\end{thm}

\begin{proof}[Proof of Theorem~\ref{thm:averaged_lag}]
	From \cite{MR1849348} we know that the drift $B_\alpha$ is a smooth spray on $\Diff_\mu^s(K)$.
	Due to the regularization operator $(\id-\alpha^2\Delta)^{-1}$, the diffusion coefficient $\Sigma_\alpha$, localised in a chart around the origin, is a $C^{1,1}$ section, by Proposition \ref{prop:VFisgood}.
	Hence the result follows from Theorem~\ref{thm:local_abstract_SDE_result}.
\end{proof}

\begin{proof}[Proof of Theorem~\ref{thm:averaged_eul}]
	The proof follows the same steps as in Section~\ref{sect: EM-SEuler}, with obvious changes.
\end{proof}

\appendix

\section{Stochastic integration on Hilbert spaces}\label{App:stochastic}

In this appendix we recall basic concepts and construction from stochastic integration on infinite-dimensional spaces and stochastic differential equations. Also here we take the main results and the approach from \cite{Elw82,BaE01}, taking also some facts from \cite{daPaZ} and \cite{BrNeVeWe2008}.

We fix a probability space $(\Omega,\mathcal{A},P)$ and a filtration $(\mathcal{F}_t)_t$ as in \ref{setup:filtration_standard}. In the following, $H$ and $E$ are separable Hilbert spaces.

\begin{setup}\label{setup:stochproccess} A stochastic process $X \colon I \times \Omega \rightarrow G$ on a time interval $I=[a,b]$ or $I=[a,+\infty)$ is a collection indexed by time of random variables $X_t\colon \Omega\rightarrow G$, $t\in I$, where $(G,\cG)$ is a measurable space. Strictly speaking, for every $t$, $X_t$ is an equivalence class, but we can also consider a modification, or version, $\tilde{X}$ of $X$: this is a map defined on $[a,t]\times \Omega$ such that, for every $t$, $\tilde{X}_t$ is a representative of $X$; we often use the same symbol for $X$ and $\tilde{X}$. A process $X$ is called 
\begin{itemize}
	\item adapted if, for each $t$, $X_t$ is $\cF_t$-measurable, and
	\item progressively measurable if, for each $t$, $X$ restricted to $[a,t]\times\Omega$ is $\cB([a,t])\otimes \cF_t$-measurable.
\end{itemize}
\end{setup}

\subsection*{Brownian motion and It\^o integral}

As a first example we recall the definition and the construction of a Brownian motion on a separable Hilbert space.

\begin{setup}[$Q$-Wiener process]\label{setup:QWiener}
Let $Q$ be a symmetric, positive semidefinite, trace-class operator on $E$. An $E$-valued $Q$-Wiener process, or $Q$-Brownian motion, with respect to $\cF$ is an $E$-valued progressively measurable process $W$ satisfying
\begin{itemize}
\item $W_0=0$ $P$-a.s.;
\item for every $s<t$, $W_t-W_s$ is a centred normal random variable with variance $(t-s)Q$ (that is, for every $v$ in $E$, $E[e^{i\lan v,W_t-W_s\ran}] = e^{-\lan v,Qv\ran/2}$);
\item for every $s<t$, $W_t-W_s$ is independent of $\cF_s$;
\item $W$ has a.s. continuous paths (that is, $P$-a.s. $t\mapsto W_t$ is continuous).
\end{itemize}
By \cite[Propositions 4.3, 4.4]{daPaZ}, for any symmetric, positive semidefinite, trace-class operator $Q$ on $E$, there exists a $Q$-Brownian motion $W$ which is given by
\begin{align*}
W_t = \sum_k \sqrt{\lambda_k} W^k_t e_k,
\end{align*}
where $(e_k)_k$ is an orthonormal basis of eigenvectors of $Q$, with non-negative eigenvalues $(\lambda_k)_k$.
\end{setup}

\begin{ex}[Trace-class Brownian motion]\label{ex:BM_trace_class}
Let $(W^k)_{k \in \N}$ be independent real Brownian motions and $\sigma_k$ be elements of $E$ such that
\begin{align*}
\sum_k \|\sigma_k\|_E^2 <\infty .
\end{align*}
Then the series $W_t = \sum_k \sigma_k W^k_t$ converges, in $L^2(\Omega)$ for every $t$ fixed and also $P$-a.s., uniformly in $t$, and it defines an $E$-valued Brownian motion, with covariance matrix $Q= \sum_k \sigma_k\sigma_k^*$, that is, $\langle Qa,b \rangle = \sum_k \langle \sigma_k,a \rangle \langle \sigma_k,b \rangle$. This result is proved in \cite[Propositions 4.3, 4.4 and Theorem 4.5]{daPaZ}, in the case of $\sigma_k$ orthogonal basis of eigenvectors of $Q$, but the proof can be extended in the same line to the case of possibly not orthogonal $\sigma_k$, as here.
\end{ex}

\begin{ex}[Cylindrical Brownian motion]
To define a cylindrical Brownian motion $\overline{W}$ on a separable Hilbert space $\overline{H}$, we set formally $\overline{W}=\sum_k W^k f_k$, where $(f_k)_k$ is a orthonormal basis of $\bar{H}$ and $W^k$ are independent real Brownian motions, defined on some filtered probability space $(\Omega,\cA,(\cF_t)_t,P)$ with complete and right-continuous filtration. Note that the above sum does not converge in $\bar{H}$.

Let now $E$ be a separable Hilbert space, such that there exists a dense embedding $i \colon \bar{H}\rightarrow E$ which is Hilbert-Schmidt, that is
\begin{align*}
\|i\|_{HS}^2:= \tr[i^*i]= \sum_k\|i f_k\|_E^2 <\infty.
\end{align*}
Then $W:=i(\overline{W})$ is an $E$-valued Brownian motion, in the sense that $\sum_k f_k W_k$ converges in $E$ as in Example \ref{ex:BM_trace_class}. See \cite[Section 4.1.2]{daPaZ} for more facts on cylindrical Brownian motion.
\end{ex}

For stochastic differential equations (SDEs) and stochastic partial differential equations  (SPDEs), one has to deal with objects of the form
\begin{align*}
\int_a^b \xi_r dW_r,
\end{align*}
where $W$ is a Brownian motion. The problem is that $W$ is not differentiable in time, hence the above integral cannot be defined as Riemann-Stieltjes integral; other extensions, for example using Young integration theory, also do not work here. However, using probability theory, we can make sense of this integral by approximation via Riemann sums and an isometry property. Here we recall the construction of this integral, called the \emph{It\^o integral}. We follow mostly \cite[Section 4.2]{daPaZ}.

In the following, $W$ is a $Q$-Wiener process on $E$ (with respect to $\cF$) as in \ref{setup:QWiener}. For simplicity of exposition, we assume for the moment that $Q$ is actually positive definite, explaining later how to extend the results to $Q$ only positive semidefinite.

\begin{defn}\label{defn:newnorm}
For $Q$ positive definite, we define a norm on $L(E,H)$:
\begin{align*}
\|A\|_{L_{2,Q}}^2 := \tr[(AQ^{1/2})(AQ^{1/2})^*].
\end{align*}
It is easy to show, using the property $\tr[(AQ^{1/2})(AQ^{1/2})^*] = \tr[Q^{1/2}A^*AQ^{1/2}]$, that
\begin{align}
\|A\|_{L_{2,Q}}\le \tr[Q]^{1/2}\|A\|_L\label{eq:L_LQ}.
\end{align}
\end{defn}
Note also that $L(E,H)$ is not complete with respect to the $L_{2,Q}$ norm.
\begin{setup}[Operator valued processes]
We denote by $\cM^Q([a,b])$ the space of progressively measurable processes $\xi\colon [a,b]\times \Omega\to L(E,H)$ such that
\begin{align*}
\|\xi\|_{\cM^Q}^2:= E \int_a^b \|\xi_t\|_{L_{2,Q}}^2 dt <\infty.
\end{align*}
It is a normed space endowed with the norm from Definition \ref{defn:newnorm}, but it is not complete. The completion will include unbounded operators, but in view of Stratonovich integration to be defined below, we do not need this. 

Here and in what follows, a $L(E,H)$-valued map is (Borel) measurable if it is measurable to $\cB(L(E,H))$. 
This applies also to the definition of progressively measurable processes \ref{setup:stochproccess} ($\cG=\cB(L(E,H))$ here). Any such $L(E,H)$-valued Borel measurable map is also measurable with respect to the Borel $\sigma$-algebra generated by the $L_{2,Q}$ norm (because the $L_{E,H}$ topology is stronger than the $L_{2,Q}$ topology); see Remark \ref{rmk:measurability_operator} later for more information on measurability for $L(E,H)$-valued maps.

We say that a process $\xi\colon[a,b]\times \Omega\to L(E,H)$ is elementary if it can be written as
\begin{align*}
\xi(t) = \sum_{j=0}^{n-1}\xi_j \one_{[t_j,t_{j+1}[}(t)
\end{align*}
where $a=t_0\le t_1< \ldots < t_n=b$ and, for every $j$, $\xi_j$ is $\cF_{t_j}$-measurable. 
Elementary processes in $\cM^Q([a,b])$ are dense in $\cM^Q([a,b])$ (see \cite[Proposition 4.22]{daPaZ} and Remark \ref{rmk:progrmeas_pred} below).
\end{setup}

For an elementary process $\xi$ in $\cM^Q([a,b])$, define the It\^o stochastic integral $I_{a,b}(\xi)$ as
\begin{align*}
I_{a,b}(\xi) := \int_a^b \xi_r dW_r :=\sum_{j=0}^{n-1}\xi_j(W_{t_{j+1}}-W_{t_j}).
\end{align*}
There hold (\cite[Proposition 4.20]{daPaZ})
\begin{align}
&E I_{a,b}(\xi) = 0,\label{eq:Ito_zero_mean}\\
&E [\|I_{a,b}(\xi)\|^2] = \int_a^b \|\xi_t\|_{L_{2,Q}}^2 dt \le C_Q \int_a^b \|\xi_t\|_{L(E,H)}^2 dt.\label{eq:Ito_isometry}
\end{align}
Hence we can extend the It\^o integral $I_{a,b}$ to a linear isometry on $\cM^Q([a,b])$, such that \eqref{eq:Ito_zero_mean} and \eqref{eq:Ito_isometry} hold for $I_{a,b}(\xi)$ for all $\xi$ in $\cM^Q([a,b])$. We use the notation (sometimes omitting $a$)
\begin{align*}
I_{a,b}(\xi) = \int_a^b \xi_r dW_r.
\end{align*}
The It\^o integral is additive on time intervals, that is $I_{[a,b]}(\xi)+I_{[b,c]}(\xi) = I_{[a,c]}(\xi)$.

For a process $\xi$ in $\cM^Q([a,b])$, we can consider the stochastic process on $[a,b]$ defined by the It\^o integral $I_{a,t}(\xi)$ for $t$ in $[a,b]$. There exists a version, still denoted by $I_t(\xi)$, of this process which is progressively measurable and is a continuous martingale, that is, for $P$-a.e. $\omega$, $t\mapsto I_t(\xi)$ is continuous as $E$-valued map and, for every $s<t$,
\begin{align*}
E[I_t(\xi)\mid \cF_s] = I_s(\xi)
\end{align*}
(in \cite{daPaZ} this is a consequence of the completeness of the space $\mathcal{M}_T^2(H)$ of continuous square integrable martingales). Moreover, the Burkholder-Davis-Gundi (BDG) inequality holds (\cite[Theorem 4.36]{daPaZ}): for every $1\le p<\infty$,
\begin{align}
E[\sup_{t\in[a,b]} \|I_t(\xi)\|^p] \le C_p E\left[\left(\int_a^b \|\xi_t\|_{L_{2,Q}}^2 dt\right)^{p/2}\right].\label{eq:BDG}
\end{align}

We also mention a property of later use: for every $\cF_a$-measurable random variable
\begin{align}
Z I_t(\xi) = I_t(Z\xi),\label{eq:F0_int}
\end{align}
as it can be verified easily for elementary processes.

Finally, all the above-mentioned facts hold for $Q$ which is only positive semidefinite, with the following change: one has to take the quotient space of $L(E,H)$ with respect to the $L_{2,Q}$ seminorm (which is not a norm in this general case).

\begin{rem}\label{rmk:progrmeas_pred}
The book \cite{daPaZ} uses a slightly different class of integrands, namely predictable processes (see \cite[Section 3.3]{daPaZ}) rather than progressively measurable ones. However one can show that, for every progressively measurable process $\xi$, there exists a predictable process $\tilde{\xi}$ with $\xi_t(\omega)=\tilde{\xi}_t(\omega)$ for $dt\otimes P$-a.e. $(t,\omega)$ (see \cite[Section 2.3, page 60]{Ku1990} for a proof in the finite-dimensional case, which works also in our case). With this identification, the results in \cite{daPaZ} hold in our context.

The paper \cite{BaE01} uses a more general setting for Banach space-valued processes and with abstract Wiener space. However, as known classically, one can reduce that setting to our case by setting $(i,\bar{H},E)$ as abstract Wiener space with $\bar{H}=Q^{1/2}E$, with the norm induced by $Q^{1/2}$, and $i\colon \bar{H}\rightarrow E$ the embedding. Note that, if $(e_k)_k$ is a complete orthonormal basis on $E$ of eigenvectors of $Q$ with eigenvalues $(\lambda_k)_k$, $Q^{1/2}e_k = \lambda_k^{1/2}e_k$ is a complete orthonormal basis on $\bar{H}$ and $W$ is a cylindrical Wiener process on $\bar{H}$.
\end{rem}

\begin{rem}\label{rmk:measurability_operator}
Concerning measurability for $L(E,H)$-valued maps, one can actually show that the Borel $\sigma$-algebra $\cB(L(E,H))$ on $L(E,H)$ generated by the $L_{E,H}$ norm and the Borel $\sigma$-algebra on $L(E,H)$ generated by the $L_{2,Q}$ norm coincide (more precisely, when $Q$ is not strictly positive, the Borel $\sigma$-algebras are taken on the quotient space and are generated by the quotient norms): the main point of the proof is to show that the $L_{E,H}$ norm is lower semi-continuous in the $L_{2,Q}$ topology. However we do not need, and we do not use, this fact here.
\end{rem}

\subsection*{It\^o processes and It\^o formula}\addcontentsline{toc}{subsection}{It\^o processes and It\^o formula}

An It\^o process is the sum of a deterministic integral and a stochastic integral. For an It\^o process, a stochastic chain rule, called the It\^o formula, holds. It differs from the classical chain rule (for smooth objects) by a second order term.

\begin{defn}
Let $H$ be a separable Hilbert space.
An $H$-valued process $\xi$ on $[a,b]$ is called an It\^o process if there exist progressively measurable processes $B\colon [a,b]\times\Omega \to H$, $S\colon [a,b]\times\Omega \to L(E,H)$, with 
\begin{align*}
\int_a^b E\|B_t\| dt <\infty, \qquad \text{ and } \qquad
\int_a^b E[\|S_t\|_{L_{2,Q}}^2] dt <\infty,
\end{align*}
such that, $P$-a.s.,
\begin{align*}
\xi_t = \xi_a +\int_a^t B_r dr +\int_a^t S_r dW_r,\quad \forall t\in[a,b].
\end{align*}
The processes $B$ and $S$ are called drift part and diffusion part of $\xi$. We write in short
\begin{align*}
d\xi_t = B_tdt +S_tdW_t.
\end{align*}
\end{defn}
For It\^o process a chain rule is available, known as It\^o formula, which contains a second order correction. The following version of It\^o formula is taken from \cite[Theorem 2.4]{BrNeVeWe2008}, see also \cite[Theorem 2.16]{BaE01}:

\begin{thm}[It\^o formula]\label{thm:Ito_formula}
Assume that $\xi$ is an $H$-valued It\^o process with drift part and diffusion part $B$ and $S$ respectively. Let $\tilde{H}$ be another separable Hilbert space and let $f\colon [0,T]\times H\to \tilde{H}$ be a function such that $f(t,\cdot)$ is twice (Fr\'echet) differentiable for each $t$ in $[a,b]$, $f(\cdot,x)$ is differentiable for each $x$ in $H$, and $f$, $\partial_t f$, $D_x f$, $D_x^2 f$ are continuous and bounded. Then it holds, $P$-a.s.\ and for all $ t\in [a,b]$,
\begin{align*}
f(t,\xi_t) &= f(a,\xi_a) +\int_a^t \partial_t f(r,\xi_r) + D_x f(r,\xi_r) B_r dr\\
&\quad+\int_a^t D_x f(r,\xi_r) S_r dW_r +\frac12 \int_a^t \tr[D^2_x f(r,\xi_r)(S_rQ^{1/2})(S_rQ^{1/2})^*] dr.
\end{align*}
\end{thm}


In short, the It\^o formula reads
\begin{align*}
df(t,\xi_t) = [\partial_t f(t,\xi_t) + D_x f(t,\xi_t) B_t] dt &+ D_x f(t,\xi_t) S_t dW_t +\\ &+\frac12 \tr[D^2_x f(t,\xi_t)(S_tQ^{1/2})(S_tQ^{1/2})^*] dt.
\end{align*}
Here, for a complete orthonormal basis $(e_k)_k$ of $E$,
\begin{align*}
\tr[D^2_x f(r,\xi_r)(S_rQ^{1/2})(S_rQ^{1/2})^*] = \sum_k D^2_x f(r,\xi_r)[S_rQ^{1/2}e_k][S_rQ^{1/2}e_k].
\end{align*}

\subsection*{Localization up to random times}\addcontentsline{toc}{subsection}{Localization up to random times}

As we will see the solution of an SDE may live only up to a finite time, which depends on the specific realization of the Brownian motion, hence this time is random. Also exit times from a given set for a process are random. Here we introduce a suitable notion of random times and its basic properties (see for example \cite[Section 3.5]{Ba2017}).

\begin{defn}
A stopping time $\tau$ is a map $\tau\colon \Omega\to I\cup\{+\infty\}$ such that, for every $t$, the set $\{\tau\le t\}$ belongs to $\cF_t$ (that is, we can decide if $\tau$ is $\le t$ or not looking only up to time $t$). We set
\begin{align*}
\cF_\tau = \{ A \in \cF_\infty \mid A\cap \{\tau\le t\}\in \cF_t\, \forall t\}
\end{align*}
the $\sigma$-algebra associated to $\tau$.

A stopping time is called accessible if there exists a sequence of stopping times $\tau_n$, non-decreasing, with $\tau_n<\tau$ $P$-a.s. for every $n$, such that $\tau_n\nearrow \tau$ $P$-a.s..
\end{defn}

Deterministic times are stopping times, the minimum and the maximum of two stopping times are also stopping times, the pointwise non-decreasing limit of stopping times is also a stopping time. A special example of stopping time is the exit time $\tau_U$ of a progressively measurable process $\xi$ from an open set $U$ in a metric space $G$ (endowed with its Borel $\sigma$-algebra), namely
\begin{align*}
\tau_U := \inf\{t\ge 0\mid \xi_t\notin U\}.
\end{align*}
If $\xi$ has continuous paths and $\xi_0$ is in $U$ $P$-a.s., then $\tau_U$ is accessible, taking $\tau_n=\tau_{U_n}$ the first exit times from $U_n$, where $U_n=\{x\in G\mid \text{dist}(x,U^c)> 1/n\}$.

Assume that the time interval is $I=[0,T]$. For a process $\xi$ in $\cM^Q([0,T])$ and a stopping time $\tau$, we have (see \cite[Lemma 4.24]{daPaZ})
\begin{align}
I_{0,\tau}(\xi) = I_{0,T}(\xi \one_{t\le \tau}).\label{eq:Itoint_stopping}
\end{align}

\begin{setup}\label{setup:Proc:stop} When the time interval is $I=[0,T]$ or $I=[0,+\infty)$, we call
\begin{align*}
[0,\tau)\times\Omega := \{(t,\omega)\in I\times\Omega \mid t< \tau(\omega)\}.
\end{align*}
A process $\xi$ on $[0,\tau)$ is a map $[0,\tau)\times \Omega\rightarrow G$, where $(G,\mathcal{G})$ is a measurable space, such that, for every $t$ in $I$, the map $\bar{\xi}_t := \xi_t \one_{\{t<\tau\}}+y_0\one_{\{t\ge\tau\}}$ is a random variable on $(\Omega,\cA,P)$, where $y_0$ is some element of $G$ (this definition of process on $[0,\tau)$ and the next definitions are independent of the choice of $y_0$); see the notation paragraph in \cite{BaE01} for a similar definition. It is called adapted, resp. progressively measurable if the process $\bar{\xi}$ is adapted, resp. progressively measurable. Analogous definitions can be given for $[0,\tau]\times \Omega$ (assuming that $\tau$ is finite $P$-a.s.) and processes defined on this space.
\end{setup}

\begin{setup}\label{setup:Proc:extension}
In the situation of \ref{setup:Proc:stop} we define spaces of processes up to stopping time.
\begin{itemize}
\item If $\tau$ is a stopping time with $\tau<\infty$ $P$-a.s., we let $\cN^Q([0,\tau])$ be the space of progressively measurable processes $\xi \colon [0,\tau]\times \Omega\to L(E,H)$ such that
\begin{align*}
\int_0^\tau \|\xi_t\|_{L_{2,Q}}^2 dt <\infty \quad P\text{-a.s.}.
\end{align*}
For $\xi$ in $\cN^Q([0,\tau])$, we can take stopping times
\begin{align*}
\tau_n=\inf\{t\in I \mid \int_0^t\|\xi_s\|^2_{L_{2,Q}}ds\ge n\}\wedge n,
\end{align*}
then $\tau_n\nearrow\tau$ and $\tau_n=\tau$ definitely $P$-a.s.. Moreover the process $\xi^n_s = \xi_s\one_{\{s\le \tau_n\}}$, $s\in I$, is in $\cM^Q(I)$ and hence we can define $I_{0,t}(\xi^n)$, $t\in I$. By \eqref{eq:Itoint_stopping}, $P$-a.s., for every $t$ in $[0,\tau]$, $I_{0,t}(\xi_n)$ is definitely constant in $n$, hence we define the It\^o integral for $\xi$ as the process on $[0,\tau]$ given by
\begin{align*}
I_t(\xi) = \lim_n I_{0,t}(\xi^n),\quad t\in [0,\tau]
\end{align*}
(setting $I_t(\xi)=0$ in the $P$-null set where the limit does not exists for some $t$).
\item If $\tau$ is an accessible (not necessarily finite) stopping time and $\tau_n<\tau$ is a non-decreasing sequence of stopping time with $\tau^n\nearrow\tau$, we call $\cN^Q_{loc}([0,\tau))$ the space of progressively measurable processes $\xi\colon [0,\tau)\times \Omega\to L(E,H)$ such that
\begin{align*}
\int_0^{\tau_n} \|\xi_t\|_{L_{2,Q}}^2 dt <\infty \quad \text{for every }n,\quad P-\text{a.s.}.
\end{align*}
For $\xi$ in $\cN^Q_{loc}([0,\tau))$, we can define $I_t(\xi)$ for $t\le \tau_n$, for all $t$, hence we can define the It\^o integral before $\tau$, $I_t(\xi)$ for $t< \tau$. By \eqref{eq:Itoint_stopping}, this integral is well-defined ($P$-a.s. for all $t$ in $[0,\tau)$) and independent of the sequence $\tau_n$ chosen.
\end{itemize}
\end{setup}
The definition of It\^o process and It\^o formula can be extended to processes defined up to a $P$-a.s.\ finite stopping time $\tau$. An $H$-valued process $\xi$ on $[0,\tau]$ is called an It\^o process if there exist progressively measurable processes $B\colon [0,\tau]\times\Omega \to H$, $S\colon[0,\tau]\times\Omega \to L(E,H)$, with $B$ in $L^1_{P-a.s.}([0,\tau])$ and $S$ in $\cN^Q([0,\tau])$, that is
\begin{align*}
&\int_0^\tau \|B_t\|_H dt <\infty \quad P-\text{a.s.},\\
&\int_0^\tau \|S_t\|_{L_{2,Q}}^2 dt <\infty \quad P-\text{a.s.},
\end{align*}
such that, $P$-a.s.,
\begin{align*}
\xi_t = \xi_a +\int_a^t B_r dr +\int_a^t S_r dW_r,\quad \forall t\in[0,\tau].
\end{align*}
For such $\xi$, It\^o formula holds as in Theorem \ref{thm:Ito_formula}, replacing $[a,b]$ with $[0,\tau]$; actually the assumption of global boundedness on $f$ and its derivatives can be removed. This extension of It\^o formula follows again from \cite[Theorem 2.4]{BrNeVeWe2008} applied to $\xi_{t\wedge \tau}\one_{\{t\le \tau\}}$.

Similarly, given $\tau$ is an accessible stopping time and $\tau_n<\tau$ is a non-decreasing sequence of stopping time with $\tau^n\nearrow\tau$, an $H$-valued process $\xi$ on $[0,\tau)$ is called an It\^o process if there exist progressively measurable processes $B\colon [0,\tau)\times\Omega \to H$, $S\colon [0,\tau)\times\Omega \to L(E,H)$, with $B$ in $L^1_{P-a.s.,loc}([0,\tau))$ and $S$ in $\cN^Q_{loc}([0,\tau))$, that is
\begin{align*}
&\int_0^{\tau_n} \|B_t\|_H dt <\infty \quad \text{for every }n,\quad P-\text{a.s.},\\
&\int_0^{\tau_n} \|S_t\|_{L_{2,Q}}^2 dt <\infty \quad \text{for every }n,\quad P-\text{a.s.},
\end{align*}
such that, $P$-a.s.,
\begin{align*}
\xi_t = \xi_a +\int_a^t B_r dr +\int_a^t S_r dW_r,\quad \forall t\in[0,\tau).
\end{align*}
Also for such $\xi$, It\^o formula holds as in Theorem \ref{thm:Ito_formula}, replacing $[a,b]$ with $[0,\tau)$ and removing the boundedness assumption on $f$ and its derivatives. This follows applying It\^o formula to $\xi_t \one_{\{t\le \tau_n\}}$ and letting $n\to \infty$.

\subsection*{Stochastic differential equations in It\^o form}\addcontentsline{toc}{subsection}{Stochastic differential equations in It\^o form}

In this subsection we give the main local well-posedness result for stochastic differential equations in It\^o form, under locally Lipschitz coefficients. We consider the SDEs only on an open set where the coefficients are Lipschitz.

Having an $E$-valued $Q$-Brownian motion $W$ as in \ref{setup:QWiener} (under the setting in \ref{setup:filtration_standard}), we consider the It\^o SDE
\begin{align}
\begin{aligned}\label{eq:SDE_Ito}
&dX_t = b(t,X_t) dt +\sigma(t,X_t) dW_t,\\
&X_0 = \zeta.
\end{aligned}
\end{align}
Here the coefficients $b\colon [0,+\infty)\times H\rightarrow H$, $\sigma \colon [0,+\infty)\times H\rightarrow L(E,H)$, called resp. drift and diffusion coefficient, are given Borel functions, the initial datum $\zeta\colon \Omega\rightarrow H$ is a $\cF_0$-measurable random variable. A progressively measurable process $X$, defined on $[0,\tau)$ for some accessible stopping time $\tau$ with $\tau>0$ $P$-a.s., is called a local strong solution of \eqref{eq:SDE_Ito}, if $b(t,X_t)$, $\sigma(t,X_t)$ are resp. in $L^1_{P-a.s.,loc}([0,\tau))$ and in $\cN^Q_{loc}([0,\tau))$ and there holds, $P$-a.s.,
\begin{align}
X_t = \zeta +\int_0^t b(r,X_r) dr +\int_0^t \sigma(r,X_r) dW_r,\quad \forall t\in[0,\tau). \label{eq:SDE_Ito_def}
\end{align}
This definition extends to progressively measurable processes $X$, defined on the closed interval $[0,\tau]$ for some $P$-a.s. finite stopping time $\tau$ with $\tau>0$ $P$-a.s., requiring that $b(t,X_t)$, $\sigma(t,X_t)$ are resp. in $L^1_{P-a.s.}([0,\tau])$ and in $\cN^Q([0,\tau])$ and that \eqref{eq:SDE_Ito_def} holds for all $t$ in $[0,\tau]$. For notational simplicity, we also consider solutions where $\tau$ can be zero on a non-zero probability set (obviously a solution defined only at the initial time does not carry useful information).

\begin{rem}
The word ``strong'' here has to be intended in the probabilistic sense and refers to the fact that the solution is defined on a given filtered probability space, with a given Brownian motion; it is not to be confused with the notion of strong solutions in the analytic sense (roughly speaking, regular solutions).
\end{rem}

\begin{thm}\label{thm:wellpos_SDE_Ito}
Let $U$ be a bounded open set in $H$ and assume that, locally uniformly in $t$, $b(t,\cdot)$ is bounded and Lipschitz on $\bar{U}$ and $\sigma(t,\cdot)$ is $L_{2,Q}$-bounded and Lipschitz on $\bar{U}$, that is, for every $T>0$,
\begin{align*}
&\|b(t,x)\|_H\le C_T,\quad \|b(t,x)-b(t,y)\|_H\le C_T\|x-y\|_H, \quad \forall x,y\in \bar{U},\,\forall t\in[0,T],\\
&\|\sigma(t,x)\|_{L_{2,Q}}\le C_T,\quad \|\sigma(t,x)-\sigma(t,y)\|_{L_{2,Q}}\le C_T\|x-y\|_H, \quad \forall x,y\in \bar{U},\,\forall t\in[0,T].
\end{align*}
Then existence and uniqueness up to the first exit time from $U$ hold, that is, for every $T>0$: there exists a solution $X$ on $[0,\tau_U\wedge T]$, where $\tau_U$ is the exit time of $X$ from $U$, of the SDE \eqref{eq:SDE_Ito} and, if $\tilde{X}$ is another solution defined on $[0,\tilde{\tau}]$, then $\tilde{X}=X$ on $[0,\tilde{\tau}\wedge\tau_U\wedge T]$ $P$-a.s.; moreover $\tau_U>0$ $P$-a.s. on the set $\{\zeta\in U\}$.
\end{thm}

The uniqueness part of the Theorem \ref{thm:wellpos_SDE_Ito} can be extended for two SDEs which coincide on $U$. For this, we consider another SDE
\begin{align*}
\begin{aligned}
&d\tilde{X}_t = \tilde{b}(t,\tilde{X}_t) dt +\tilde{\sigma}(t,\tilde{X}_t) dW_t,\\
&\tilde{X}_0 = \tilde{\zeta},
\end{aligned}
\end{align*}
with $\tilde{b}$ and $\tilde{\sigma}$ given Borel functions and $\tilde{\zeta}$ $\cF_0$-measurable.

\begin{lem}\label{lem:uniq_SDE_Ito}
Let $U$ be a bounded open set in $H$. Assume that, locally uniformly in $t$, $b(t,\cdot)$, $\tilde{b}(t,\cdot)$ coincide on $\bar{U}$ and are bounded and Lipschitz on $\bar{U}$ and $\sigma(t,\cdot)$, $\tilde\sigma(t,\cdot)$ coincide on $\bar{U}$ and are $L_{2,Q}$-bounded and Lipschitz on $\bar{U}$; assume also that $\zeta$ and $\tilde{\zeta}$ coincide on some $\cF_0$-measurable set $\Omega_0$. Let $X$, $\tilde{X}$ be two solutions, on $[0,\tau]$ and $[0,\tilde{\tau}]$ resp., of the SDEs driven by $b$, $\sigma$ and $\tilde{b}$, $\tilde{\sigma}$ resp.. Then $X$ and $\tilde{X}$ coincide on $[0,\tau\wedge \tilde{\tau}\wedge \tau_U]\times\Omega_0$ $P$-a.s., where $\tau_U$ is the first exit time of $X$ (or $\tilde{X}$) from $U$.
\end{lem}

\begin{proof}
We call $\rho = \tau\wedge \tilde{\tau}\wedge \tau_U$. The difference $(X_{t\wedge \rho}-\tilde{X}_{t\wedge \rho})\one_{\Omega_0}$ satisfies the equation
\begin{align*}
(X_{t\wedge \rho}-\tilde{X}_{t\wedge \rho})\one_{\Omega_0} &= \int_0^t [b(r,X_r)-b(r,\tilde{X}_r)]\one_{r\le \rho}\one_{\Omega_0} dr\\
&\quad +\int_0^t [\sigma(r,X_r)-\sigma(r,\tilde{X}_r)]\one_{\{r\le \rho\}}\one_{\Omega_0} dW_r,
\end{align*}
where we have used \eqref{eq:F0_int} and \eqref{eq:Itoint_stopping}. We take the expectation of the supremum of the squared norm: by H\"older inequality for the drift and BDG inequality \eqref{eq:BDG} for the diffusion, we get
\begin{align*}
E\sup_{t\in[0,T]} \|X_{t\wedge \rho}-\tilde{X}_{t\wedge \rho}\|^2 \one_{\Omega_0} &\le C \int_0^T E \|b(r,X_r)-b(r,\tilde{X}_r)\|^2 \one_{\{r\le \rho\}}\one_{\Omega_0} dr\\
&\quad +C \int_0^t E\|\sigma(r,X_r)-\sigma(r,\tilde{X}_r)\|_{L_2,Q}^2 \one_{\{r\le \rho\}}\one_{\Omega_0} dr.
\end{align*}
By the Lipschitz property of $b$ and $\sigma$ on $U$, we obtain
\begin{align*}
E\sup_{t\in[0,T]} \|X_{t\wedge \rho}-\tilde{X}_{t\wedge \rho}\|^2 \one_{\Omega_0} \le C \int_0^T E \sup_{t\in[0,r]} \|X_{t\wedge \rho}-\tilde{X}_{t\wedge \rho}\|^2 \one_{\Omega_0} dr
\end{align*}
(the constant $C$ possibly depending on $T$). We apply Gronwall inequality and get
\begin{align*}
E\sup_{t\in[0,T]} \|X_{t\wedge \rho}-\tilde{X}_{t\wedge \rho}\|^2 \one_{\Omega_0} \le 0,
\end{align*}
that is uniqueness.
\end{proof}

\begin{proof}[Proof of Theorem \ref{thm:wellpos_SDE_Ito}]
Uniqueness is a particular case of Lemma \ref{lem:uniq_SDE_Ito}. For existence, we take, for any positive integer $n$,
\begin{align*}
U_n = \{ x\in H \mid \text{dist}(x,U^c)> 1/n \}\subseteq U.
\end{align*}
For each $n$ we deduce from \cite[15.9 (2)]{MR1471480} that there exists a map $\varphi_n\colon H\rightarrow \R$, $C^1$ with both $\varphi_n$ and its derivative globally Lipschitz continuous (whence bounded), such that $\varphi_n\equiv 1$ on $U_n$ and $\varphi_n\equiv 0$ on $U^c$. Therefore the coefficients
\begin{align*}
b^n(t,x) = b(t,x)\varphi_n(x),\quad \sigma^n(t,x) = \sigma(t,x)\varphi_n(x)
\end{align*}
are globally Lipschitz and globally bounded on $H$, uniformly in $t$ in $[0,T]$. Hence the SDE on $H$
\begin{align*}
dX^n_t = b^n(t,X^n_t)dt +\sigma^n(t,X^n_t) dW_t
\end{align*}
admits a (unique) solution $X^n$, by classical well-posedness theory (see e.g. \cite[Theorem 7.2]{daPaZ}, applied with $A=0$ there). For any $m>n$, since $b^m=b^n$ and $\sigma^n=\sigma^m$ on $U_n$, by Lemma \ref{lem:uniq_SDE_Ito} $X^m$ and $X^n$ coincide up to the first exit time $\tau_{U_n}$ from $U^n$. Hence, calling $\tau_U$ the limit of the non-decreasing sequence $\tau_{U_n}$, $P$-a.s., for each $t<\tau_U$, the sequence $X^n_t \one_{\{t\le \tau_n\}}$ is definitely constant, in particular the limit
\begin{align*}
X_t = \lim_n X^n_t,\quad t<\tau_U
\end{align*}
is well-defined $P$-a.s. and $\tau_U$ is the (possibly infinite) exit time of $X$ from $U$. Using \eqref{eq:Itoint_stopping}, one can check easily that $X$ is a solution to \eqref{eq:SDE_Ito} on $[0,\tau_U)$. Finally, since $U$ is bounded and $b$ and $\sigma$ are bounded, if $\tau_U$ is finite, then $X$ is extended by continuity to $\tau_U$ and this extension is still a solution to \eqref{eq:SDE_Ito}. This shows existence for the SDE \eqref{eq:SDE_Ito}. The fact that $\tau_U>0$ $P$-a.s. on $\{\zeta\in U\}$ follows from the continuity of paths of $X$. The proof is complete.
\end{proof}

Note that, by \eqref{eq:L_LQ}, if the diffusion coefficient $\sigma$ is Lipschitz in the $L(E,H)$ norm on some set $U$ (locally uniformly in $t$), then it is also Lipschitz in the $L_{2,Q}$ norm (locally uniformly in $t$), hence the previous results apply.

\begin{rem}\label{rmk:Ito_random_drift}
The definition of solution to an SDE and the existence and uniqueness Theorem \ref{thm:wellpos_SDE_Ito} can be extended to the case of a random drift, under the following assumption: given an accessible stopping time $\tau$, the drift $b\colon [0,\tau)\times\Omega\times H \rightarrow H$ is such that
\begin{itemize}
\item for every $x$ in $\bar{U}$, $b(\cdot,\cdot,x)$ is progressively measurable, and
\item it holds $P$-a.s.: for every $t$ in $[0,\tau)$, $b(t,\omega,\cdot)$ is Lipschitz continuous on $\bar{U}$, uniformly with respect to $(t,\omega)$.
\end{itemize}
Then existence and uniqueness for the SDE \eqref{eq:SDE_Ito} hold on $[0,\tau_U\wedge \tau)$. The proof proceeds as in the proof of Theorem \ref{thm:wellpos_SDE_Ito} (noting that \cite[Theorem 7.2]{daPaZ} applies also to random drifts).
\end{rem}

\subsection*{Stratonovich integral and Stratonovich differential equations}\addcontentsline{toc}{subsection}{Stratonovich integral and Stratonovich differential equations}

Here we recall the Stratonovich stochastic integral and Stratonovich SDEs. While not a martingale, the Stratonovich integral has the advantage of a classical chain rule, without second order terms, hence it is more suited to be extended to the manifold case. The price to pay is an additional degree of regularity for well-posedness of SDEs driven by Stratonovich integrals.

\begin{defn}Given an $E$-valued $Q$-Brownian motion $W$, \ref{setup:QWiener} (under the setting in \ref{setup:filtration_standard}) and a $P$-a.s. finite stopping time $\tau$, we introduce the space $\cN([0,\tau])$ of $L(E,H)$-progressively measurable processes such that
\begin{align*}
\int_0^\tau \|\xi_t\|_{L(E,H)}^2 dt <\infty\quad P-\text{a.s.}.
\end{align*}
\end{defn}

Let $\xi$ be an It\^o process on $[0,\tau]$ of the form
\begin{align*}
d\xi_t = B_t dt +S_t dW_t,
\end{align*}
where $B$ is progressively measurable and in $L^1_{P-a.s.}([0,\tau])$ and $S$ is in $\cN([0,\tau])$ (note that this requirement is stronger than just $S$ in $\cN^Q([0,\tau])$).

\begin{defn}
Let $\xi$ be as before and let $g\colon H\to L(E,H')$ be a $C^1$ map, where $H'$ is another separable Hilbert space. The Stratonovich stochastic integral of $g(\xi)$ in $dW$ is the $H'$-valued progressively measurable process on $[0,\tau]$ defined by
\begin{align}
\int_0^t g(\xi_r) \bullet d W_r := \int_0^t g(\xi_r) dW_r +\frac12 \int_0^t \tr[g'(\xi_r)S_rQ] dr, \quad t\in[0,\tau].\label{eq:Strat_def}
\end{align}
where, for $(e_k)_k$ complete orthonormal basis of $E$,
\begin{align*}
\tr[g'(\xi_r)S_rQ] = \sum_k [g'(\xi_r)S_rQ^{1/2}e_k][Q^{1/2}e_k].
\end{align*}
\end{defn}

For $x,v$ in $H$, $g'(x).v$ is in $L(E,H')$, so the It\^o-Stratonovich correction, that is the last term in \eqref{eq:Strat_def}, makes sense. It is well-defined and progressively measurable as
\begin{align*}
H\times L(E,H)\ni (x,S) \mapsto \tr[g'(x)SQ]\in H'
\end{align*}
is continuous: indeed, taking $(e_k)_k$ basis of eigenfunctions of $Q$ with eigenvalues $\lambda_k$, we have
\begin{align*}
\sum_k \|[g'(x)SQ^{1/2}e_k][Q^{1/2}e_k]\|_{H'}&\le \|g'(x)\|_{L(H,L(E,H'))} \|S\|_{L(E,H)} \sum_k\lambda_k\\
&= \|g'(x)\|_{L(H,L(E,H'))} \|S\|_{L(E,H)} \tr[Q]
\end{align*}
and similarly
\begin{align*}
\sum_k \|[(g'(x)-g'(y))SQ^{1/2}e_k][Q^{1/2}e_k]\|_{H'}\le \|g'(x)-g'(y)\|_{L(H,L(E,H'))} \|S\|_{L(E,H)} \tr[Q].
\end{align*}

\begin{rem}\label{rem:Strat_def}
By this very definition, the Stratonovich integral is defined in terms of $\xi$, $g$ and $S$ separately: even if $g(\xi)=h(\eta)$ for another function $h$ and It\^o process $\eta$, the definition for the Stratonovich integral of $h(\eta)$ may give a different object. Actually one can show that the integral depends only on $g(\xi)$, regardless the way it is written, using the concept of quadratic variation (see also \cite[Theorem 3.7]{BaE01} without using quadratic variation). But we will not use this fact and understand the Stratonovich integral as function of $\xi$, $g$ and $S$.
\end{rem}

For $\tau$ accessible stopping time, the above definition can be extended to It\^o processes $\xi$ on $[0,\tau)$ in the form
\begin{align*}
d\xi_t = B_t dt +S_t dW_t,
\end{align*}
where $B$ is progressively measurable and in $L^1_{P-a.s.,loc}([0,\tau))$ and $S$ is in $\cN_{loc}([0,\tau))$, the space of progressively measurable processes with
\begin{align*}
\int_0^{\tau_n} \|S_t\|_{L(E,H)}^2 dt <\infty\quad \text{for every }n,\quad P-\text{a.s.}.
\end{align*}

Now we consider Stratonovich stochastic differential equations, in this case autonomous, namely
\begin{align}
\begin{aligned}\label{eq:SDE_Strat}
&dX_t = b(X_t)dt +\sigma(X_t) \bullet dW_t,\\
&X_0 = \zeta.
\end{aligned}
\end{align}
Also here the drift $b\colon H\rightarrow H$ and the diffusion coefficient $\sigma\colon H\rightarrow L(E,H)$ are given functions assumed continuous and in $C^1$ resp., the initial datum $\zeta\colon \Omega\rightarrow H$ is a $\cF_0$-measurable random variable. A progressively measurable process $X$, defined on $[0,\tau)$ for some accessible stopping time $\tau$ with $\tau>0$ $P$-a.s., is called a local strong solution if it has $P$-a.s. continuous paths and there holds, $P$-a.s. for all $t\in[0,\tau)$
\begin{align}
X_t = \zeta +\int_0^t b(X_r) dr +\int_0^t \sigma(X_r) dW_r +\frac12 \int_0^t \tr[\sigma'(X_r)\sigma(X_r)Q] dr . \label{eq:SDE_Strat_def}
\end{align}
This definition extends to progressively measurable processes $X$, defined on the closed interval $[0,\tau]$ for some $P$-a.s. finite stopping time $\tau$ with $\tau>0$ $P$-a.s., requiring that \eqref{eq:SDE_Strat_def} holds for all $t$ in $[0,\tau]$. Moreover, if $X$ takes values $P$-a.s. in an open subset $U$ of $H$, it is enough that $b$ and $\sigma$ are defined and in $C^0$, in $C^1$ resp. on $U$.

Note that, since $b$ and $\sigma$ are resp. in $C$ and in $C^1$, all the integrals are well-defined. Note also that, if $X$ is a solution, then $X$ is an It\^o process and the last two terms in \eqref{eq:Strat_def} are the Stratonovich integral
\begin{align*}
\int_0^t \sigma(X_r) \bullet dW_r = \int_0^t \sigma(X_r) dW_r +\frac12 \int_0^t \tr[\sigma'(X_r)\sigma(X_r)Q] dr.
\end{align*}

We can extend the above definitions and properties to Stratonovich differentials, namely, on $[0,\tau)$ (or $[0,\tau]$),
\begin{align*}
dX = Bdt +\sigma(X_t) \bullet dW_t,
\end{align*}
where $\sigma$ is $C^1$ and $B$ is progressively measurable and in $L^1_{P-a.s.,loc}([0,\tau))$\\ (or in $L^1_{P-a.s.}([0,\tau])$): that is, $X$ satisfies
\begin{align}
X_t = X_0 +\int_0^t B_r dr +\int_0^t \sigma(X_r) dW_r +\frac12 \int_0^t \tr[\sigma'(X_r)\sigma(X_r)Q] dr, \quad \forall t\in[0,\tau).\label{eq:diff_Strat_def}
\end{align}

\begin{thm}[It\^o formula for Stratonovich differentials]\label{thm:Ito_formula_Strat}
Assume that $X$ satisfies \eqref{eq:diff_Strat_def}, with $B$, $\sigma$ as above. Let $f\colon H\to \tilde{H}$ be a $C^2$ function, where $\tilde{H}$ is another Hilbert space. Then it holds, $P$-a.s.,
\begin{align}
f(X_t) &= f(\zeta) +\int_0^t D_x f(X_r) B_r dr +\int_0^t D_x f(X_r) \sigma(X_r) \bullet dW_r, \quad \forall t\in [0,\tau).\label{eq:Ito_formula_Strat}
\end{align} 
\end{thm}

\begin{proof}
We apply It\^o formula \ref{thm:Ito_formula}, extended to possibly unbounded $C^2$ functions and to processes defined up to $\tau$, as seen in the localization subsection \ref{setup:Proc:extension}, to the It\^o process $X$ in \eqref{eq:diff_Strat_def} and to $f$:
\begin{align*}
df(X_t) &= D_xf(X_t)B_t dt +D_xf(X_t)\sigma(X_t) dW_t +\frac12 D_xf(X_t) \tr[\sigma'(X_t)\sigma(X_t)Q] dt +\\
&\quad +\frac12 \tr[D^2_xf(X_t)(\sigma(X_t)Q^{1/2})(\sigma(X_t)Q^{1/2})^*] dt
\end{align*}
To prove \eqref{eq:Ito_formula_Strat}, we have to show that
\begin{align}
\begin{aligned}\label{eq:Ito_form_proof}
\tr[D_x(D_xf \sigma)(X_t)\sigma(X_t)Q] &= D_xf(X_t) \tr[\sigma'(X_t)\sigma(X_t)Q]\\
&\quad + \tr[D^2_xf(X_t)(\sigma(X_t)Q^{1/2})(\sigma(X_t)Q^{1/2})^*].
\end{aligned}
\end{align}
Now we note that $D_xf \sigma$ is in $C^1(H,L(E,H))$ and, for any $x,v$ in $H$, it holds
\begin{align*}
D_x(D_xf \sigma)(x)[v] = D_xf(x)[\sigma'(x) v] +D^2_xf(x)[v][\sigma(x)]
\end{align*}
as equality in $L(E,H)$. We take $x=X_t$, $v=\sigma(X_t)Q^{1/2}e_k$ and compose with $Q^{1/2}e_k$: we obtain, for each $k$, $P$-a.s. for every $t<\tau$,
\begin{align*}
D_x(D_xf \sigma)(X_t)[\sigma(X_t)Q^{1/2}e_k][Q^{1/2}e_k] &= D_xf(X_t)[\sigma'(X_t)\sigma(X_t)Q^{1/2}e_k][Q^{1/2}e_k]\\
&\quad +D^2_xf(x)[\sigma(X_t)Q^{1/2}e_k][\sigma(X_t)Q^{1/2}e_k].
\end{align*}
Summing over $k$, we get \eqref{eq:Ito_form_proof}. The proof is complete.
\end{proof}

The local well-posedness result requires the diffusion coefficient $\sigma$ to have one more degree of regularity (with respect to the analogue result in the It\^o context):

\begin{thm}\label{thm:wellpos_SDE_Strat}
Let $U$ be a bounded open set in $H$ and assume that $b$ is in $C^{0,1}(\bar{U},H)$, that is Lipschitz on $\bar{U}$, and $\sigma$ is in $C^{1,1}(\bar{U},L(E,H))$, that is $C^1$ from $\bar{U}$ to $L(E,H)$, with $\sigma'$ Lipschitz from $\bar{U}$ to $L(H,L(E,H))$. Then existence and uniqueness up to the first exit time from $U$ hold, that is, for every $T>0$: there exists a solution $X$ on $[0,\tau_U\wedge T]$, where $\tau_U$ is the exit time of $X$ from $U$, of the SDE \eqref{eq:SDE_Strat} and, if $\tilde{X}$ is another solution defined on $[0,\tilde{\tau}]$, then $\tilde{X}=X$ on $[0,\tilde{\tau}\wedge\tau_U\wedge T]$ $P$-a.s.; moreover $\tau_U>0$ $P$-a.s. on the set $\{\zeta\in U\}$.
\end{thm}

As for the It\^o SDE, we can extend uniqueness to two SDEs which coincide on $U$. For this, we consider the Stratonovich SDE
\begin{align*}
\begin{aligned}
&d\tilde{X}_t = \tilde{b}(X_t)dt +\tilde{\sigma}(X_t) \bullet dW_t,\\
&\tilde{X}_0 = \tilde{\zeta}.
\end{aligned}
\end{align*}

\begin{lem}\label{lem:uniq_SDE_Strat}
Let $U$ be a bounded open set in $H$. Assume that $b$, $\tilde{b}$ coincide on $\bar{U}$ and are in $C^{0,1}$ on $\bar{U}$ and $\sigma$, $\tilde\sigma$ coincide on $\bar{U}$ and are $C^{1,1}$ (as $L(E,H)$-valued maps) on $\bar{U}$; assume also that $\zeta$ and $\tilde{\zeta}$ coincide on some $\cF_0$-measurable set $\Omega_0$. Let $X$, $\tilde{X}$ be two solutions, on $[0,\tau]$ and $[0,\tilde{\tau}]$ resp., of the SDEs driven by $b$, $\sigma$ and $\tilde{b}$, $\tilde{\sigma}$ resp.. Then $X$ and $\tilde{X}$ coincide on $[0,\tau\wedge \tilde{\tau}\wedge \tau_U]\times\Omega_0$ $P$-a.s., where $\tau_U$ is the first exit time of $X$ (or $\tilde{X}$) from $U$.
\end{lem}

\begin{proof}[Proof of Theorem \ref{thm:wellpos_SDE_Strat} and Lemma \ref{lem:uniq_SDE_Strat}]
The equation \eqref{eq:SDE_Strat} reads in It\^o form:
\begin{align*}
dX = \left(b(X)+\frac12\tr[\sigma'(X)\sigma(X)Q]\right)dt +\sigma(X)dW.
\end{align*}
Hence the results follow from Theorem \ref{thm:wellpos_SDE_Ito} and Lemma \ref{lem:uniq_SDE_Ito} once we show that the above coefficients satisfy the Lipschitz assumptions for the It\^o SDEs. The diffusion coefficient $\sigma$ is $C^{1,1}$ on $\bar{U}$, in particular $\sigma'$ is Lipschitz, hence bounded on $\bar{U}$ because $\bar{U}$ is bounded; therefore $\sigma$ is also Lipschitz on $\bar{U}$ and so bounded, in the $L(E,H)$ topology. Since the $L(E,H)$-topology is stronger than the $L_{2,Q}$ topology, $\sigma$ is also bounded and Lipschitz in the $L_{2,Q}$ topology. The function $b$ is Lipschitz, and so bounded, on $\bar{U}$, by assumption. It remains to show that the It\^o-Stratonovich correction $\tr[\sigma'\sigma Q]$ is Lipschitz on $\bar{U}$. For every $x,y$ in $\bar{U}$, taking $(e_k)_k$ basis of eigenvectors of $Q$ with eigenvalues $\lambda_k$, we have
\begin{align*}
&\|\tr[(\sigma'(x)\sigma(x)-\sigma'(y)\sigma(y)) Q]\|_H \le \sum_k \|[(\sigma'(x)\sigma(x)-\sigma'(y)\sigma(y))Q^{1/2}e_k][Q^{1/2}e_k]\|_H\\
\le& \sum_k \left(\|[(\sigma'(x)-\sigma'(y))\sigma(x)Q^{1/2}e_k\|_{L(E,H)} \right. + \\ & \left. \hphantom{\sum_k (\|[(\sigma'(x)-\sigma'(y))}+\|[\sigma'(y)(\sigma(x)-\sigma(y))Q^{1/2}e_k\|_{L(E,H)}\right) \|Q^{1/2}e_k\|_E\\
\le& \sum_k \left(\|\sigma'(x)-\sigma'(y)\|_{L(H,L(E,H))}\|\sigma(x)Q^{1/2}e_k\|_H + \right.\\ & \left.\hphantom{\sum_k (\|[(\sigma'(x)-\sigma'(y))} +\|\sigma'(y)\|_{L(H,L(E,H))}\|(\sigma(x)-\sigma(y))Q^{1/2}e_k\|_H\right) \|Q^{1/2}e_k\|_E\\
\le& \sum_k \left(\|\sigma'(x)-\sigma'(y)\|_{L(H,L(E,H))}\|\sigma(x)\|_{L(E,H)} + \right. \\ & \left. \hphantom{\sum_k (\|[(\sigma'(x)-\sigma'(y))} +\|\sigma'(y)\|_{L(H,L(E,H))}\|\sigma(x)-\sigma(y)\|_{L(E,H)}\right) \|Q^{1/2}e_k\|_E^2\\
\le& C\|x-y\| \tr[Q],
\end{align*}
where in the last inequality we used that $\sigma$ and $\sigma'$ are Lipschitz and bounded on $\bar{U}$. Hence the It\^o-Stratonovich correction is Lipschitz, and so bounded, on $\bar{U}$. 
\end{proof}

\begin{rem}\label{rmk:Strat_random_drift}
As for the It\^o SDEs, the definition of solution to a Stratonovich SDE and the existence and uniqueness Theorem \ref{thm:wellpos_SDE_Strat} can be extended to the case of a random drift, under the same assumptions of Remark \ref{rmk:Ito_random_drift}.
\end{rem}

\section{Essentials on spaces of Sobolev maps as infinite-dimensional manifolds}\label{app: Sobolev}

In this appendix we recall the construction of manifolds of mappings. 
Here we follow the classical exposition of \cite{EM70,MR0248880} and recall the relevant constructions.
Let us begin by recalling the construction of Sobolev type sections of vector bundles

\begin{setup}
In addition to our usual requirements we set $d \coloneq \text{dim } K$ and fix a Riemannian metric $g_K$ on $K$. Its associated volume form is denoted by $\mu$.
Now $\pi_E \colon E \rightarrow K$ will be a smooth vector bundle over $K$ together with a Riemannian structure $g_E$ on $E$.
Further, we let $N$ be a finite-dimensional manifold without boundary together with a Riemannian metric $g_N$ and Riemannian exponential map $\exp_N \colon TN  \supseteq \Omega \rightarrow N$. Again $H^s(K,N)$ denotes the space of $H^s$-maps $f\colon K \rightarrow N$ (cf.\ Definition \ref{defn: nonlinear:Sobolev}).
\end{setup}

Albeit we chose auxiliary Riemannian structures, the constructions in this appendix do not depend on the specific choice. In particular, every choice yields the same space of Sobolev sections (see e.g.\ \cite[\S 8]{MR0248880}) we are about to define now.

\subsection*{Sobolev sections of a vector bundle}\addcontentsline{toc}{subsection}{Sobolev sections of a vector bundle}

We recommend the survey in \cite[Appendix B]{MR2030823} on spaces of Sobolev sections. For the readers convenience, the necessary definitions and main results are repeated now.

\begin{setup}[Spaces of Sobolev sections]\label{setup: Sobolevsect}
Denote by $L^2 (E)$ the set of all Borel-measurable sections $X$ of $\pi_E\colon E\rightarrow M$ such that 
\begin{align}
\lVert X\rVert_{L^2} = \left(\int g_E (X(k),X(k)) \mathrm{d} \mu\right)^{\frac{1}{2}} < \infty. \label{L2norm}
\end{align}
Note that $L^2(E)$ neither depends on the choice of $\mu$ nor on the choice of $g_E$ (cf.\ \cite[p.25]{MR0248880}).
For $s\in \N$ we recall from \cite[\S 2]{MR0248880} (cf.\ \cite[\S 1]{MR583436}) that taking (truncated) Taylor expansions in charts gives rise to the $s$-jet bundle $J^s (E)$. Denoting by $\Gamma^s (E)$ the $C^s$-sections of $E$, there is a continuous linear map $j^s \colon \Gamma^s (E) \rightarrow \Gamma^0 (J^s (E))$, the $s$\emph{-jet extension}.\footnote{The mapping $j^s(X)$ sends $X$ to the family of iterated tangent maps $T_k^i X$, for $i\in \N_0, i\leq s$ and $k\in K$.}.
For $s > \frac{d}{2}$ we define now the \emph{space of $H^s (E)$ (Sobolev) sections} as the completion of the space 
$$\Gamma^{s,2}(K,E) \coloneq \{ X \in \Gamma^s (E) \mid j^s (X) \in L^2 (J^s(E))\}$$
with respect to the norm \eqref{L2norm}. If $E=TM$ is the tangent bundle, we write $\mathfrak{X}^s(K) \coloneq H^s (TM)$ for the \emph{space of $H^s$-Sobolev vector fields}. 

It is clear from the construction, that $H^s(E)$ is a Banach space (whose norm sums up the $L^2$-norms \eqref{L2norm} of the jets $j^s(X)$) and even a Hilbert space (whose inner product is induced by the $L^2$ inner product on the jet spaces).
\end{setup}

Contrary to our treatment of $H^s$-morphisms between manifolds (which were defined as being $H^s$ in suitable charts), we defined the Sobolev sections as a completion with respect to an $L^2$-inner product. This immediately establishes the Hilbert space property, but lacks a convenient description in local charts. It is well known \cite[Section 2]{EM70}, \cite[Appendix B]{MR2030823} that instead one can also define Sobolev sections using local trivialisations, being mindful of the warning \ref{warning} we need to take some care in showing that indeed $H^s (E) \subseteq H^s (K,E)$.

\begin{setup}[Localisation in charts]\label{setup: loc:Sob:sect}
 Let $\pi_E \colon E \rightarrow K$ be a vector bundle with typical fibre $F$. Due to the compactness of $K$ we can choose a finite set $\{x_k\} \subseteq K$ together with an atlas $(U_k,\kappa_k)_{1\leq k \leq \ell}$ of charts and an atlas of bundle trivialisation $\{\Psi_k\}$ of $E$ such that 
 \begin{enumerate}
 \item $\overline{U}_k \subseteq \pi_E (\text{dom} \kappa_k)$,
 \item $\kappa_k (U_k) = B_R (\kappa_k (x_k)) \cap \overline{\R}^d_+$,\footnote{Here the closed half space $\overline{\R}^d_+$ in $\R^d$ is needed as we allow $K$ to have smooth boundary, cf.\ \cite{MR2954043}. If $x_k \not \in \partial K$, we have $B_R (\kappa_k (x_k)) \subseteq \overline{\R}^d_+$, while for $x_k \in \partial K$ we may assume $\kappa_k(x_k) = 0$.} where $B_R(\kappa_k (x_k))$ is the open $R$-ball around $\kappa_k (x_k)$, for some $R >0$, 
 \item the operator norms of $\{T^r_u \Psi_k, T^r_u \kappa_k\}_{u\in U_k, 1\leq r \leq s, 1\leq k \leq \ell}$ are uniformly bounded (w.r.t.\ $g_E,g_K$).
 \item for any $1\leq i,j\leq \ell$ the boundaries of $\kappa_i(U_i \cap U_j)$ are piecewise smooth. 
 \end{enumerate}
 To construct such an atlas $(U_k,\kappa_k)$, one easily adapts the construction \cite[Lemma 3.1]{IKT13} (using the mapping $f \colon K \rightarrow \R , x \mapsto 0$, shrinking the charts obtained there to ensure that they are contained in a bundle trivialisation). Note that the cited construction was only carried out in the case of manifolds without boundary. However, the construction carries over to the smooth boundary case, by using a Riemannian metric adapted to the boundary, cf.\ \ref{setup:boundary}.
 
 Denote by $H^s (\kappa_k (U_k), F)$ the completion of $C^\infty (\kappa_k (U_k), F)$\footnote{If $\partial K \neq \emptyset$ the choice of charts and the Whitney extension theorem, cf.\ e.g.\ \cite{1801.04126v4}, entail that smooth functions on $\kappa (U)$ are exactly the restrictions of smooth functions on $B_R (\kappa_k (x_k))$.} with respect to 
\begin{equation}\label{eq: local:norm}
\langle f,g\rangle_{H^s, U_k} \coloneq \sum_{|\alpha| = r \in \N_0, r\leq s} \int_{\kappa_k (U_k)} \langle D^\alpha f(x), D^\alpha g (x)\rangle \, \mathrm{d} x,
\end{equation}
where $D^\alpha$ is the iterated partial derivative with respect to a multiindex $\alpha$.\footnote{Note that $B_R (\kappa(x)) \cap \overline{\R}^d_+$ is non-compact, whence a new definition of the $H^s$-space is needed.} The space $H^s (\kappa_k (U_k),F)$ coincide with the usual $H^s$-Sobolev space, cf.\ \cite[4.2]{MR2453959} and \cite[4.5]{MR1163193}. In the boundary case, we can invoke the Calder\'{o}n extension theorem \cite[\S 3 Theorem 2]{MR0198494} and \cite{MR0324726} to see that $H^s (B_R (0) \cap \overline{\R}^d_+,F)$ also coincides with the usual Sobolev space.

Now \cite[Remark B.1]{MR2030823} implies that $X \in H^s (E)$ if and only if the principal part of the representative $\Psi_k \circ X \circ \kappa_k^{-1}$ is contained in $H^s (\kappa_k (U_k),F)$. Moreover, loc.cit.\ shows that the sum of the inner products \eqref{eq: local:norm} is equivalent to the inner product inducing the Hilbert space structure of $H^s(E)$. As a consequence of \eqref{eq: local:norm}, cf.\ \cite[Eq.\ (69) and Section 4.1]{IKT13},
the mapping 
\begin{equation}\label{eq: loc:iso}
H^s(E) \rightarrow \prod_{1\leq k \leq \ell} H^s (\kappa_k (U_k), F),\quad X \mapsto (\text{pr}_2 \circ \Psi_k \circ X \circ \kappa_k^{-1})_{1\leq k\leq \ell},
\end{equation}
induces a Hilbert space isomorphism of $H^s(E)$ onto its image.
\end{setup}

\begin{lem} \label{lem: char:Sobosect}
For $s > \frac{d}{2}$, $K$ compact (possibly with smooth boundary) and $\pi \colon E \rightarrow K$ a finite rank bundle. We have 
$$H^s (E) = \{ f \in H^s (K,E) \mid \pi \circ f = \id_K\}.$$
Moreover, $H^s(E)$ is separable.
\end{lem}
\begin{proof}
We have already seen in \ref{setup: loc:Sob:sect} that a section is in $H^s(E)$ if and only if localises in charts an $H^s$-Sobolev map. Hence $H^s(E) \subseteq H^s(K,E)$. For the converse take an element $X$ in $H^s(K,E)$ which is also a section of $\pi_E$. Restricting the bundle trivialisations to a relatively compact neighbourhood $O$ of the image of $X$ we obtain a fine cover $((U_k),(\text{dom } \Psi_k \cap O)_k,X)$ of $X$ (see \cite[Definition 3.2]{IKT13}. Then \cite[Lemma 3.2]{IKT13} implies that $X$ localises in these charts to an $H^s$-mapping, whence $X \in H^s (E)$ by the above. This proves the first statement of the Lemma.

To prove that $H^s(E)$ is separable, let us assume first that $\partial K = \emptyset$. Then the isomorphism \eqref{eq: loc:iso} identifies $H^s(E)$ with a subspace of a finite product of Sobolev spaces on open balls in euclidean space. These Sobolev spaces are well known to be separable Hilbert spaces \cite[Theorem 3.6]{MR2424078}, whence $H^s(E)$ is separable as a subspace of a metrizable separable space \cite[4.1.16 Corollary]{Eng89}. 
If $\partial K \neq \emptyset$ we embed $K$ into its double $\tilde{K}$ and note that the smooth vector bundle $E \rightarrow K$ extends to a smooth vector bundle $\tilde{E} \rightarrow \tilde{K}$ by \cite[X: Theorem 5]{MR0198494}. Due to the Calderon extension theorem, the restriction map $H^s (\tilde{E}) \rightarrow H^s (\tilde{E}\mid_K) \cong H^s(E)$ is a continuous surjective mapping \cite[X: Theorem 7 (Restriction theorem)]{MR0198494}. We deduce that $H^s(E)$ is separable as the continuous image of the separable space $H^s(\tilde{E})$ (\cite[1.4.11 Corollary]{Eng89}.
\end{proof}

\begin{rem}\label{rem: IKT}
In \cite{IKT13} Sobolev sections are described via the characterisation in Lemma \ref{lem: char:Sobosect}. This leads to a natural notion of Sobolev mappings allowing to treat fractional Sobolev exponents on manifolds without boundary (see e.g.\ \cite{MR3635359} for a discussion). Though we followed the older approach in \cite{MR0248880}, the proofs of Lemma \ref{lem: char:Sobosect} and the approach in \ref{setup: loc:Sob:sect} follows the characterisation in local charts.
\end{rem}

\subsection*{Manifolds of $H^s$-mappings}\addcontentsline{toc}{subsection}{Manifolds of Sobolev mappings}
We now endow the set $H^s (K,N)$ with a manifold structure. To this end let us first consider spaces of sections covering an $H^s$-map.
\begin{defn}\label{def:sectsp}
Consider $f \in H^s (K,N)$ and define 
$$H^s_f (K,TN) \coloneq \{X \in H^s (K,TN) \mid \pi_N \circ X = f\}.$$
We endow $H^s_f (K,TN)$ with the unique Hilbert space structure turning the obvious bijection $H^s_f (K,TN) \cong H^s (f^*TN)$ into an isomorphism of Hilbert spaces.
\end{defn}

\begin{setup}[Canonical $H^s$-charts]\label{setup: charts}
With the help of the Riemannian exponential mapping (and shrinking $\Omega$ if necessary), we obtain a diffeomorphism
$E \coloneq (\pi_N, \exp_N) \colon TN \supseteq \Omega \rightarrow E(\Omega) \subseteq N \times N$
onto an open neighbourhood $E(\Omega)$ of the diagonal in $N \times N$. Shrinking $\Omega$ we may assume that $E(\Omega)$ is symmetric with respect to interchanging the components of $N\times N$.
Define for $f \in H^s (K,N)$ the set 
$$\mathcal{U}_f \coloneq \{g \in H^s (K,N) \mid (f,g) \in E(\Omega)\}.$$
together with a map  
$$\varphi_f \colon \mathcal{U}_f \rightarrow H^s_f (K,TN),\quad g \mapsto E^{-1} (f,g).$$
Clearly, the inverse of this mapping is 
$$\varphi_f^{-1} \colon H_f^s (K,TN) \supseteq O_f \rightarrow H^s (K,N) , g \mapsto E \circ g,$$
where $O_f = \{g \in H^s_f (K,TN)\mid g (K) \subseteq \Omega\}$. 
We shall now assume that $f\in C^\infty (K,N)$. In light of \ref{setup: loc:Sob:sect}, we can invoke the Sobolev embedding theorem \cite[Corollary after theorem 9.2]{MR0248880} to see that the topology on $H_f^s (K,TN)$ is finer than the topology induced by the compact-open topology via the inclusion $H^s_f(K,TN) \subseteq H^s(K,TN) \subseteq C^0 (K,TN)$. Thus $O_f$ is open in the Hilbert space $H^f_s (K,TN)$. We call $(\varphi_f,\mathcal{U}_f)$ a \emph{canonical chart} around $f \in C^\infty (K,N)$.

We claim now that the domains of the maps $\{\varphi_f\}_{f\in C^\infty (K,N)}$ cover $H^s (K,N)$. To this end recall that, due to our choice of $s$, $H^s(K,N) \subseteq C^0(K,N)$. As smooth mappings are dense in $ C^0(M,N)$ (with respect to the compact open topology), every $C^0$ neighborhood of an $H^s$-map contains a smooth map. Choosing a suitable neighborhood, we find for every $g \in H^s (K,N)$ a suitable $f\in C^\infty (K,N) \cap \mathcal{U}_g$. Since $E(\Omega)$ is symmetric, we deduce that $g \in \mathcal{U}_f$. Since the domains of the charts $\{\varphi_f\}_{f\in C^\infty (K,N)}$ cover $H^s (K,N)$ we can endow $H^s (K,N)$ with the identification topology induced by all canonical charts. 
One can easily check that the identification $H_f^s (K,N) \cong H^s (f^*TN)$ identifies the change of charts $\varphi_g \circ \varphi_f^{-1}$ with the postcomposition $F_* \colon H^s (\Omega_{f,g}) \rightarrow H^s (g^*TN)$, where $F$ is a (smooth) fibre preserving map. Hence for $f,g \in C^\infty (K,N)$ the change of charts is smooth due to \cite[Theorem 13.4]{MR0248880}. In particular, the canonical charts are homeomorphisms onto their image and thus form a $C^\infty$-atlas for the Hilbert manifold $H^s(K,N)$. 

Further, one identifies the smooth curves into the manifold of Sobolev morphisms. This allows one to identify tangent space and the tangent manifold as
\begin{equation}
T_f H^s (K,N) = H^s_f (K,TN) \cong H^s (f^*TN) \quad TH^s (K,N) = H^s (K,TN).\label{tangent:ident}
\end{equation}
Moreover, this construction is compatible with the natural choice of charts. We refer to \cite[Section 2 and Appendix A]{BaHaM19} for a detailed account.
\end{setup}

Note that the manifold topology on $H^s(K,N)$ coincides with the Sobolev $H^s$-topology see \cite{MR1636569,MR1163193}. This is proved for example in \cite[Section 3]{IKT13}. A priori it is not clear from our construction that the manifold topology on $H^s(K,N)$ is Hausdorff. However, it will follow directly from the following lemma:

\begin{lem}\label{lem: Sobolevemb}
For $\ell\geq 0$ the inclusion
$\iota_\ell \colon H^{s+\ell} (K,N) \rightarrow C^\ell (K,N)$ is smooth.
\end{lem}

\begin{proof}
The map $\iota_\ell$ makes sense as every $H^{s+\ell}$ map is $C^\ell$ due to the Sobolev embedding theorem.
Now, we just need to note that the canonical charts for $C^\ell (K,N)$ are constructed similarly to the ones for $H^{s+\ell} (K,N)$, the only difference being that they are defined on spaces of $C^\ell$ sections (cf.\ e.g.\ \cite[Appendix A]{1811.02888}). 
Hence the $\iota_\ell$ conjugates in canonical charts to the inclusion $\text{Sob}_\ell \colon H^{s+\ell} (f^*TN) \rightarrow C^\ell (f^*TN)$ which is continuous linear (whence smooth) due to the Sobolev embedding theorem \cite[Corollary after theorem 9.2]{MR0248880}.
\end{proof}

In general, the composition of $H^s$-maps will not yield an $H^s$-map, so composition might be ill defined. To remedy this we will require from now on $s > \frac{d}{2} + 1$ and work instead of all $H^s$-morphisms with the $H^s$-diffeomorphisms:
$$\Diff^s (K) := \{\Phi \in H^s (K,K) \mid \Phi \text{ is bijective and }\Phi^{-1} \in H^s (K,K) \}$$ 
Sticking with our approach we will assume first that $K$ has no boundary. Then we discuss the necessary changes for the general case. This distinction is only relevant for historic reasons, taking a more elaborate approach would yield similar results in the boundary case.\footnote{\cite[l.-7 p.109]{EM70} states that "$H^s(K,K)$ is not a manifold; it has infinite-dimensional corners". However, as \cite{MR583436} proves this presents no problem for $C^k$-maps, and the $H^s$-statement is similar.}

\subsubsection*{Case 1: The underlying manifold has no boundary}

If $K$ has no boundary, one can prove (cf.\ \cite[p.107]{EM70} or \cite{IKT13} for a modern reference) that $\Diff^s (K)$ is an open subset of $H^s (K,K)$ and composition is well-defined and turns $\Diff^s (K)$ into a topological group. 
Moreover, it is known (see \cite{IKT13}) that the composition is differentiable on certain subspaces:

\begin{setup}[Differentiability properties of the composition]\label{setup: HL}
The composition map 
$$\text{Comp} \colon H^{s+\ell} (K,N) \times \Diff^s (K) \rightarrow H^{s} (K,N), (\zeta,\Phi) \mapsto \zeta \circ \Phi$$
is a $C^\ell$-mapping for all $\ell \in \N_0$.
Analysing this further, one can prove that for $\Phi\in \Diff^s (K)$ the right multiplication
$$R_\Phi \colon \Diff^{s}(K) \rightarrow \Diff^s (K),\quad \xi \mapsto \xi \circ \Phi$$
is smooth, and left composition with $\zeta \in H^{s+\ell}(K,N)$, $\ell \in \N$
$$L_\zeta \colon \Diff^s (K) \rightarrow H^{s} (K,N),\quad \Phi \mapsto \zeta \circ \Phi$$
is only a a $C^\ell$-map. Using the identification $TH^s (K,N) \cong H^s (K,TN)$ in \eqref{tangent:ident} the derivatives can be identified as 
\begin{align*}
TR_\Phi (X) = X \circ \Phi,  \qquad TL_\zeta (\eta) = T\zeta \circ \eta = L_{T\zeta} (\eta).
\end{align*}
Thus $\Diff^s (K)$ is only a topological group, but not a Lie group. In particular, the inversion map $\iota \colon \Diff^{s} (K) \rightarrow \Diff^s (K)$ is only continuous. Considering inversion as a mapping $\text{inv} \colon \Diff^{s+\ell}(K) \rightarrow \Diff^s (K)$ it is of class $C^\ell$ with first derivative (for $\ell \geq 1$) given by the formula \cite[p. 108]{EM70} (cf.\ \cite[Proposition 11.13]{MR583436})
$$T_\Phi \text{inv} (\eta) = - (T\Phi^{-1}) \circ \eta \circ \Phi^{-1}, \quad \eta \in T_\Phi \Diff^s (K)$$
The topological group $\Diff^s(K)$ is a Hilbert manifold such that right multiplication is smooth, whence $\Diff^s (K)$ is a so called half Lie group, see \cite{MR3836195}.
\end{setup}

\subsubsection*{Case 2: The underlying manifold has smooth boundary}
We discuss now the case of $K$ having smooth boundary and follow \cite{EM70} (though a direct approach as in \cite{MR583436} yields the same).

\begin{setup}\label{setup:boundary}
 Embed $K$ into its double $\tilde{K}$ \cite[Theorem 6.3]{MR0198479}. We can endow $\tilde{K}$ with a Riemannian metric $\tilde{g}$ such that $\partial K \subseteq \tilde{K}$ becomes a totally geodesic submanifold \cite[Lemma 6.4]{EM70}.
 Constructing the canonical manifold structure on $H^s (K,\tilde{K})$ with respect to $\tilde{g}$, the canonical $H^s$-chart around $\Phi \in \Diff^s (K)$ induces a submanifold chart for $\Diff^s (K)$ mapping the closed subspace
 $$\mathfrak{X}^s_{\Phi, \partial} (K) \coloneq \{X \in H^s (K,T\tilde{K}) \mid \pi \circ X=\Phi, \quad X(k) \in T_k \partial K,\  \forall k \in \partial K\}$$
 to $\Diff^s (K) \subseteq H^s (K,\tilde{K})$ \cite[Lemma 6.6]{EM70}.
 We note the following important facts:
 If $K$ has smooth boundary, $\Diff^s (K)$ is a closed submanifold of $H^s (K,\tilde{K})$ and the canonical charts with respect to the Riemannian metric $\tilde{g}$ restrict to submanifold charts. As a consequence a direct calculation shows
 \begin{enumerate}
 \item The differentiability properties of the composition maps from \ref{setup: HL} carry over. In particular, $\Diff^s (K)$ is a half Lie group.
 \item Applying the identification \eqref{tangent:ident} to the composition map on the submanifold $\Diff^s (K)$, the identities from \ref{setup: HL} for the derivatives of the  composition mappings are available.
 \end{enumerate}
 \end{setup}

The point of the above definition is that we can use the Hodge theory developed in \cite[Section 7]{EM70} to translate the Euler equation on manifolds with boundary to our infinite-dimensional setting. To this end consider the group of volume preserving diffeomorphisms (with respect to the volume form $\mu$, which for $K$ with boundary is assumed to be induced by the restriction of the Riemannian metric $\tilde{g}$ in \ref{setup:boundary}).

\begin{setup}[Volume preserving $H^s$ diffeomorphisms]\label{setup: volpres}
The $H^s$-diffeomorphisms for $s > \tfrac{d}{2} +1$ act by pullback on the differential forms on $K$. Hence we can consider the subgroup 
$$\Dmu:= \{\Phi \in \Diff^s (K) \mid \Phi^* \mu = \mu\}$$
of \emph{volume preserving $H^s$-diffeomorphisms}. It is well known that $\Dmu$ is a closed submanifold of $\Diff^s (K)$ \cite[Theorem 4.2 and Theorem 8.1]{EM70}, whence it is a half Lie group and in particular, a topological group. Furthermore, $\Dmu$ is a Hilbert manifold modelled on a separable Hilbert space. We can thus apply the Birkhoff-Kakutani theorem \cite[Theorem 8.3]{HaR79} to see that $\Dmu$ is metrizable. Since the model space is a separable Hilbert space, we see that $\Dmu$ has a connected, separable (whence second countable \cite[Corollary 4.1.16]{Eng89}) open identity neighbourhood $U$. Whence the identity component $\Diff_\mu^s (K)_\circ = \bigcup_{n\in \N} U^n$ \cite[Theorem 7.4]{HaR79} (and thus every component) of $\Dmu$ is second countable, whence separable.\footnote{Our proofs for the separability and metrizability of the manifold $\Dmu$ (resp.\ its components) heavily exploited that $\Dmu$ is a topological group. If one wants to avoid using such arguments, a direct proof using either the strong Riemannian metric (see below) or working with the Sobolev $H^s$-topology on $H^s(K,K)$ will yield similar results.} 
\smallskip

\textbf{The tangent bundle of the volume preserving $H^s$-diffeomorphisms}
As in the finite-dimensional setting we can construct a tangent bundle $T\Dmu$ of $\Dmu$ \cite[III. \S 2]{MR1666820}. Since $\Dmu$ is metrizable, so is $T\Dmu$ (using \cite[Proposition 29.4]{MR1471480}, $T\Dmu$ is paracompact and locally metrizable, whence metrizable by the Smirnov metrization theorem \cite[Ex. 5.4.A]{Eng89}). Moreover, as the identity component $\Dmu_\circ$ is open and separable, the same holds for $T\Dmu_\circ \subseteq T\Dmu$ (Note that $T\Dmu$ is second countable as in any trivialisation $T\Dmu_\circ$ is homeomorphic to a product of second countable spaces (the model space is separable!) and we can cover $T\Dmu$ with a countable cover of such neighbourhoods as $\Dmu$ is second countable.).
The tangent space at the identity can be identified as the space $$\mathfrak{X}^s_\mu(K):= \{ X \in \mathfrak{X}^s(K) \mid \text{div} (X) = 0, X(k) \in T_k \partial K, \ \forall k \in \partial K\}$$ of divergence free $H^s$-vector fields (where we suppress in the notation the condition of being tangential to the boundary). Due to the loss of derivatives in forming the Lie bracket of vector fields, $\mathfrak{X}^s_\mu(K)$ is not a Lie algebra. 
For $s \in \N$ the manifold $\Dmu$ possess a strong right invariant Riemannian metric\footnote{Recall that a Riemannian metric is strong, if it induces the topology on every tangent space. While strong metrics retain properties of finite-dimensional Riemannian metrics, also new phenomena occur, see \cite{MR1666820,MR1330918} for introductions to strong Riemannian metrics in infinite dimensions.}. To this end, one defines an inner product on $\mathfrak{X}^s_\mu(K)$ via
$$\langle X, Y\rangle_{H^s,L} := \int g_K(X, L \circ Y)\mathrm{d} \mu,$$
where $L = (\id + \Delta^s)$ and $\Delta u = (\delta \mathrm{d}u^\flat + \mathrm{d} \delta u^\flat)^\sharp$ is the positive definite Hodge Laplacian. Then 
$$g_{H_s} (X_\varphi ,Y_\varphi) := \langle X_\varphi \circ \varphi^{-1} , Y_\varphi \circ \varphi^{-1}\rangle_{H^s,L}, \quad X_\varphi, Y_\varphi \in T_\varphi \Dmu^s(K)$$ is a smooth right invariant Riemannian metric \cite[p.140]{EM70}. Note that smoothness of the metric is actually quite surprising as the inversion map is only continuous (see \cite[Section 6.1]{MR3635359} for a detailed discussion).
\end{setup}

\addcontentsline{toc}{section}{References}
\bibliography{EM_SPDE}

\end{document}